\documentclass[10pt]{amsart}

\usepackage{amsmath}
\usepackage{amssymb}
\usepackage{amsfonts}
\usepackage{geometry}
\usepackage{url}

\usepackage{amsmath}
\usepackage{amssymb}
\usepackage{amsfonts}
\usepackage{geometry}
\usepackage{url}
\usepackage{setspace}
\usepackage{amsthm}
\usepackage[all,cmtip]{xy}
\usepackage{amsmath,amscd}
\usepackage{amsthm}

\newcommand{\gtw}{G_{12}} 
\newcommand{\h}{\mathfrak{h}}
\newcommand{\rp}[2]{\mathbf{#1_{#2}}}
\newcommand{\M}[2]{M_{#1}(#2)} 
\newcommand{\LL}[2]{L_{#1}(#2)}
\newcommand{\mb}[2]{[M_{#1}(#2)]} 
\newcommand{\lb}[2]{[L_{#1}(#2)]}
\newcommand{\ch}[2]{\chi_{\rp{#1}{#2}}}
\newcommand{\chr}{\operatorname{ch}} 
 
\newcommand{\Hom}{\operatorname{Hom}}

\newcommand{\GL}[1]{\operatorname{GL}_{#1}}
\newcommand{\mc}{\mathbb{C}} 
\newcommand{\mq}{\mathbb{Q}}

\newcommand{\mz}{\mathbb{Z}}

\newcommand{\mo}{\mathcal{O}}
\newcommand{\vphi}{\varphi}
\newcommand{\rar}{\rightarrow}

\newcommand{\mf}[1]{\mathfrak{#1}}

\newcommand{\It}[1]{\textit{#1}}
\newcommand{\Bf}[1]{\textbf{#1}}

\newtheorem{thm}{Theorem}[section]
\newtheorem{lemma}[thm]{Lemma}
\newtheorem{proposition}[thm]{Proposition}
\newtheorem{cor}[thm]{Corollary}

\theoremstyle{definition}
\newtheorem{definition}[thm]{Definition}

\begin{document}

\title{Category $\mo$ for the Rational Cherednik Algebra associated to the complex reflection group $G_{12}$}

\author{Martina Balagovi\' c \and Christopher Policastro}
\address{Department of Mathematics,  Massachusetts Institute
of Technology, Cambridge, MA 02139, USA }
\email{martinab@math.mit.edu, cpoli@mit.edu}
\maketitle

\begin{abstract}
In this paper, we describe the irreducible representations in category $\mo$ of the rational Cherednik algebra $H_{c}(G_{12},\mathfrak{h})$ associated to the complex reflection group $\gtw$ with reflection representation $\mathfrak{h}$ for an arbitary complex parameter $c$. In particular, we determine the irreducible finite dimensional representations and compute their characters. 
\end{abstract}

\section{Introduction}

 The rational Cherednik algebra $H_{c}(W,\mathfrak{h})$ associated to a complex reflection group $W$ with reflection representation $\h$ is a certain infinite dimensional, non-commutative, associative algebra over $\mc$. It can be viewed as a deformation of the algebra $\mc[W] \ltimes S(\h^{*} \oplus \h)$, depending on finitely many complex parameters $c_{s}$ which correspond to conjugacy classes of reflections in $W$. 
 
Little can be said about arbitrary representations of rational Cherednik algebras. However, one can define  category $\mo$ for  $H_{c}(W,\h)$ (sometimes called $\mo_{c}$ or $\mo_{c}(W,\mathfrak{h})$)% \cite{GGOR}
, which contains better understood representations. It can be shown that category $\mo$ is generated by certain standard (or Verma) modules $M_{c}(\tau)$, which are parametrized by irreducible representations $\tau$ of the finite group $W$. There is a certain contravariant form  on highest weight modules in category $\mo$, that we call $B$, and that is an analogue of the Shapovalov form in Lie theory. It has the property that its kernel on $M_{c}(\tau)$ is the maximal proper submodule; hence the quotient $L_{c}(\tau)=M_{c}/\ker B$ is an irreducible representations with an inherited nondegenerate form on them. In fact, this construction gives all the irreducible representations in category $\mo_{c}$. The basic and in general still open question is to determine the structure of these modules, for example by giving their characters, or equivalently by giving their description in the Grothendieck group of $\mo_{c}$ in terms of modules $M_{c}(\tau)$.

%While the representation theory of rational Cherednik algebras of complex reflection groups has been seriously studied in recent years, and results or partial results exist, much remains unknown. For groups of type $A$ this question was considered in [BEG] and for dihedral groups in [CH]; for $G_{4}$, in the Shephard-Todd notation [ST], it was studied in [S].       

In this paper we study the rational Cherednik algebra associated to a particular complex reflection group, $G_{12}$ in Shephard-Todd notation% \cite{ST}
. As $G_{12}$ has just one conjugacy class of reflections, $c$ is in our case a complex constant. We describe the irreducible representations in $\mo_{c}(G_{12},\mathfrak{h})$. More precisely, we determine the values of $c$ for which the category is semisimple, find Grothendieck group expressions for irreducible modules in terms of standard modules, and the characters for each irreducible finite dimensional representation. The main theorem of the paper is Theorem \ref{main}.

The structure of the proof is as follows. It is known that there exists a functor $\text{KZ}_{c}$ from $\mo_c$ to the category of representations of a Hecke algebra $\mathcal{H}_{q}(W)$ associated to the complex reflection group $W$. We first determine semisimplicity conditions on $\mathcal{H}_{q}(W)$ in terms of the parameter $q$, and then use the $\text{KZ}_{c}$ functor to determine the semisimplicity conditions on $\mo_{c}$ in terms of the parameter $c$. It is a well known fact that this can always be done, and we do it for $G_{12}$ using the CHEVIE package of the computer algebra software GAP. We find that $\mo_{c}(G_{12},\mathfrak{h})$ is semisimple unless $c=m/12, \, m\in \mathbb{Z}, \, m \equiv 1,3,4,5,6,7,8,9,11 (\operatorname{mod} 12)$ (see Theorem \ref{semisimple}).

Next, there is a simple equivalence of categories $\mo_{c}\to \mo_{-c}$, enabling us to assume $c>0$. Also well known but more involved are the equivalences of categories  $\mo_{1/d}\to \mo_{r/d}$ for $r,d>0,\, d\ne 2$, and $r$ and $d$ relatively prime, and $\mo_{c}\to \mo_{c+1}$. These equivalences $\Phi_{c,c'}:\mo_{c}\to \mo_{c'}$ map standard modules to standard modules and irreducible modules to irreducible modules. It is explained in \cite{GG} how to calculate the permutation $\varphi_{c,c'}$ of the irreducible representations of $W$ that realizes it, in the sense that $\Phi_{c,c'}(M_{c}(\tau))=M_{c'}(\varphi_{c,c'}(\tau))$ and $\Phi_{c,c'}(L_{c}(\tau))=L_{c'}(\varphi_{c,c'}(\tau))$. We calculate this permutation in case $W=G_{12}$. As a consequence of this, it is enough to describe the irreducible representations $L_{c}(\tau)$ in terms of standard representations $M_{c}(\tau)$ for $c \in \left\{1/12, 1/4, 1/3, 1/2 \right\}$, as we can use the equivalences of categories to get descriptions for all other values of $c$. 

Finally, for $c \in \left\{1/12, 1/4, 1/3, 1/2 \right\}$ we describe the irreducible representations in $\mo_{c}$ case by case, using a variety of methods such as computations of the contravariant form $B$ in the algebra software MAGMA, basic representation theory of the finite group $G_{12}$ and induction and restriction functors for rational Cherednik algebras.

The methods described in this paper should be easily applicable to other complex reflection groups in the case of equal parameters.

The organization of the paper is as follows. In section 2, we give the basic definitions and constructions that will be needed for the statement of results. These include among other things the definition of rational Cherednik algebras, category $\mo$, standard and irreducible modules, characters of representations and Grothendieck group. In section 3, we state the main results: theorems \ref{semisimple} and \ref{main} and corollaries \ref{korolar} and \ref{aspherical}. In section 4 we give the proof of theorem \ref{semisimple} with all the prerequisites for it (facts about Hecke algebras and KZ functors). In section 5 we describe several equivalences of categories between categories $\mo_{c}$ for different $c$ that reduce the proof of theorem \ref{main} to the cases $c \in \left\{1/12, 1/4, 1/3, 1/2 \right\}$. Character formulas that accompany these equivalences allow us to, assuming theorem \ref{main}, give a proof of corollary \ref{aspherical} at the end of this section. Section 6 then describes tools that we are going to use in section 7, and section 7 gives the computations for the remaining values $c \in \left\{1/12, 1/4, 1/3, 1/2 \right\}$, thus finishing the proof of theorem \ref{main}. 
 
\subsection*{Acknowledgements} We are very grateful to Pavel Etingof for suggesting the problem and devoting his time to it through many helpful conversations. We would like to thank Maria Chlouveraki and Gunter Malle for their explanations of cyclotomic Hecke algebras. This work was done with the support of SPUR, MIT's program for undergraduate research, where C.P. was a participant. The work of M.B. was partially supported by the NSF grant DMS-0504847.

\section{Definitions and Notation}

 In this section, we recall the properties of rational Cherednik algebras and their representations necessary for the statement of our results. A more detailed description using the same notation, along with proofs of the basic properties omitted here, can be found in \cite{EM}.  

\subsection{Rational Cherednik Algebra} 
 Let $\h$ denote a finite dimensional complex vector space. We call a nontrivial diagonalizable element of $\operatorname{GL}(\h)$ a \It{reflection} if it fixes a codimension 1 subspace, and a finite subgroup of $\operatorname{GL}(\h)$ a \It{complex reflection group} if it is generated by reflections; in that case $\mathfrak{h}$ is called its \It{reflection representation}. Let $W$ be such a group and $S$ its set of all reflections in $W$. 
 
Let $\mathfrak{h}^*$ denote the dual space to $\mathfrak{h}$, and $(\cdot,\cdot):\mathfrak{h}^*\times \mathfrak{h}\to \mc$, $(\cdot,\cdot):\mathfrak{h}\times \mathfrak{h}^*\to \mc$  the canonical pairings. We recall that $W$ acts naturally on $\h^*$ by the dual representation, and that an element $s\in W$ is a reflection on $\mathfrak{h}$ if and only if it is a reflection on $\mathfrak{h}^*$. For such  an element $s \in S$, let $\lambda_s$ be its nontrivial eigenvalue in $\h^*$, let $\alpha_s \in \h^{*}$ be its eigenvector corresponding to the eigenvalue $\lambda_s$, and let $\alpha^{\vee}_{s} \in \h$ be the eigenvector corresponding to the eigenvalue $\lambda_s^{-1}$. These eigenvectors are unique up to scaling, so we may assume they are mutually normalized so that $(\alpha_s,\alpha^{\vee}_{s})=2$. 

Let  $c: S \to \mc$ be a conjugation invariant map, meaning a map such that $c(wsw^{-1})=c(s)$ for all $s\in S, w\in W$. 

\begin{definition}
The rational Cherednik algebra $H_{c}(W,\h)$ is the quotient of $\mc[W] \ltimes T(\h \oplus \h^{*})$ by the relations $$[x,x']=0,   [y,y']=0,   [y,x]=(y,x) - \sum_{s \in S}{c(s)(y,\alpha_s)(x,\alpha^{\vee}_{s}) s}$$ for all $x,x' \in \h^*$, and $y,y' \in \h$. 
\end{definition}

\subsection{PBW basis} 
Let $\left\{x_1,\ldots,x_n \right\}$ be any basis of $\h^*$ and $\left\{y_1,\ldots,y_n \right\}$ any basis of $\h$. It is clear from the definition that elements of the form $wx_{1}^{p_1} \cdots x_{n}^{p_n}y_{1}^{q_1} \cdots y_{n}^{q_n}$, for $w \in W$, $p_{i},q_{i} \in \mathbb{Z}_{\ge 0}$, span $H_{c}(W,\h)$. It can be shown that they in fact form a basis. In other words, there is an isomorphism of vector spaces $$H_{c}(W,\h) \cong \mc[W] \otimes S\h^{*} \otimes S\h .$$ 

From now on, fix $\left\{y_1,\ldots,y_n \right\}$ a basis of $\h$ and $\left\{x_1,\ldots,x_n \right\}$ a dual basis of $\h^*$.

\subsection{Grading} 
Define a grading on $H_{c}(W,\h)$ by $\operatorname{deg}(w)=0$ for $w \in W$, $\operatorname{deg}(x)=1$ for $x \in \h^{*}$, and $\operatorname{deg}(y)=-1$ for $y \in \h$. This grading is inner, in the sense that there exists a grading element $$\Bf{h}= \sum_{i=1}^{n}{x_{i}y_{i}} + \frac{1}{2}\operatorname{dim}(\h) - \sum_{s \in S}{ \frac{2c(s)}{1-\lambda_{s}} s } \in H_{c}(W,\h)$$ such that $[\Bf{h},x]=x$, $[\Bf{h},y]=-y$, and $[\Bf{h},w]=0$, so the $n$-th graded piece of $H_{c}(W,\h)$ is exactly the $n$-eigenspace of $\Bf{h}$.  

 \subsection{Category $\mo$} The category of representations of $H_{c}(W,\h)$ that we are interested in, called category $\mo$ (or possibly $\mo_{c}$ or $\mo_{c}(W,\mathfrak{h})$ to recall the dependence on $c, W, \mathfrak{h}$), has as objects those $H_{c}(W,\h)$-modules that are finitely generated over $S\h^{*}$ and locally nilpotent over $S\h$. It is a full subcategory of $H_{c}(W,\h)$-$\operatorname{Mod}$, closed under subquotients and extensions.
 
It can be shown that $\mo$ contains all finite dimensional $H_{c}(W,\h)$-modules and that every object in $\mo$ is a direct sum of finite dimensional generalized eigenspaces of $\Bf{h}$. Since $[\Bf{h},w]=0$, each such eigenspace is a finite dimensional representation of $W$.

 \subsection{Standard (Verma) Modules} Let $\tau$ denote an irreducible representation of $W$. Let \linebreak[4] $\mc[W] \ltimes S \h \subset H_{c}(W,\h)$ act on $\tau$ by requiring that each noninvertible element of $S \h$ annihilates $\tau$. The standard module $\M{c}{\tau}$ (sometimes called Verma module) is then defined as the induced representation $$\M{c}{\tau}=H_{c}(W,\h) \otimes_{\mc[W] \ltimes S \h} \tau \; .$$ 

So, as a vector space, $\M{c}{\tau} \cong S\h^{*} \otimes \tau$. It is easy to write the action of $w\in W$, $x\in \h$ and $y\in \h^*$ after this isomorphism. Namely, let  $m_{i}\in \mathbb{Z}_{\ge 0}$, $z\in \tau$, and write the action of $w$ on $x\in \h^*$ as $w.x$. Then the actions are given by 
$$x_{i}=\left(x_{1}^{m_1}\ldots x_{n}^{m_{n}}\otimes z\mapsto x_{1}^{m_1}\ldots x_{i}^{m_i+1}\ldots x_{n}^{m_{n}}\otimes z  \right)$$
$$w=\left(x_{1}^{m_1}\ldots x_{n}^{m_{n}}\otimes z\mapsto (w.x_{1})^{m_1}\ldots (w.x_{n})^{m_{n}}\otimes (w.z)  \right)$$
$$y_{i}=D_{y_{i}}= \partial_{i} \otimes \operatorname{id} - \sum_{s \in S}{\frac{2 c(s)}{1-\lambda_{s}} \frac{(\alpha_{s},y_{i})}{\alpha_{s}}(1-s) \otimes s } .$$

The operators $D_{y_i}$ are called Dunkl operators. 

It is clear that the modules $M_{c}(\tau)$ are in category $\mo$. We can also characterize them by the expected universal mapping property. If $N$ is a $H_{c}(W,\h)$-module which contains a representation of $W$ isomorphic to $\tau$, and such that $S\h$ acts on this representation as zero, then there exists a unique map from $\M{c}{\tau}$ to $N$ that restricts to the identity on $\tau$.    

 \subsection{$\bf{h}$-weights}
 Standard modules are graded representations of $H_{c}(W,\h)$. The graded pieces of $M_{c}(\tau)$ are $S^{i}\h^* \otimes \tau$. It is clear that $x_{i}\in \h^*$ map $S^{i}\h^* \otimes \tau$ to $S^{i+1}\h^* \otimes \tau$, that $y_{i}\in \h$ map $S^{i}\h^* \otimes \tau$ to $S^{i-1}\h^* \otimes \tau$, and that $w\in W$ preserve it. Hence, this gives a grading on $M_{c}(\tau)$ compatible to the one we defined on $H_{c}(W,\h)$.
 
 Because the grading on  $H_{c}(W,\h)$ was inner, it is convenient to label the graded pieces by the eigenvalues of the same element $\bf{h}$. Notice that on $\tau\subset M_{c}(\tau)$ the grading element $\bf{h}$ acts as
 $$h_{c}(\tau):=\frac{1}{2}\operatorname{dim}(\h) - \sum_{s \in S}{ \frac{2c(s)}{1-\lambda_{s}}s},$$
which is a constant because $\tau$ is irreducible and $c$ is conjugation invariant. Then it follows that  the grading element $\bf{h}$ acts on $S^{i}\h^*\otimes \tau$ as $$h_{c}(\tau)+i.$$
Call these eigenvalues $\Bf{h}$-weights, and label $S^{i}\h^*\otimes \tau=M_{c}(\tau)[h_{c}(\tau)+i].$ All the submodules and quotients of $M_{c}(\tau)$ inherit this grading. 

 \subsection{Irreducible Modules in Category $\mo$} Using the grading, one can show that $M_{c}(\tau)$ has a unique maximal proper submodule (which might be $0$). Call it $J_{c}(\tau)$. The quotient module $L_{c}(\tau)=M_{c}(\tau)/J_{c}(\tau)$ is irreducible. It is not hard to see that all the irreducible modules in category $\mo$ are of this form. In particular, there are only finitely many of them, as they are labeled by irreducible representations of $W$. 

Using the grading again and the fact that $\tau$ is irreducible, it is easy to see that standard modules are indecomposable. In particular, category $\mo_{c}$ is semisimple if and only if all standard modules in it are simple, ie. $M_{c}(\tau)=L_{c}(\tau)$ for all $\tau$.

We want to understand these modules: for which $c$ and $\tau$ are they finite dimensional, how fast do their graded pieces grow if they are infinite dimensional, how do they decompose as $W$ representations. These questions received a lot of attention in the last years, and many partial results exist, but the answers are not known in general. This paper focuses on answering them for a particular case $W=G_{12}$.

 \subsection{Characters} Let $N$ be an object of $\mo$. As noted, there is a decomposition $N=\oplus_{\alpha}N[\alpha]$ where $N[\alpha]$ is a generalized eigenspace of the grading element $\Bf{h}$. For a complex variable $t$ and $w \in W$, we define the character of $N$ as the formal series $$\chr_{N}(t,w)=\sum_{\alpha}{\operatorname{Tr}_{N[\alpha]}(w) \cdot t^{\alpha} } \; .$$ Here $\operatorname{Tr}_{N[\alpha]}(w)$ denotes the trace of the operator $w$ on the finite dimensional space $N[\alpha]$. Let the group character of $W$ on $\tau$ be $\chi_{\tau}(w)=\operatorname{Tr}_{\tau}(w)$. For $N=\M{c}{\tau}$ the above defined character is equal to  $$\chr_{\M{c}{\tau}}(t,w)=\frac{\chi_{\tau}(w) \cdot t^{h_{c}(\tau)}}{\operatorname{det}_{\h^{*}}(1-tw)} .$$ 

 \subsection{Grothendieck Group}  Let $K(\mo_{c})$ be the Grothendieck group of $\mo_{c}$, in other words the abelian group with generators $[N], \, N\in \mo_c$, and relations $[N]=[N']+[N'']$ for each short exact sequence $0 \rar N' \rar N \rar N'' \rar 0 $ in $\mo_c$. 
 
It can be shown that modules in $\mo_{c}$ have finite length. As all irreducible modules are of the form $L_{c}(\tau)$, it follows that elements $[L_{c}(\tau)]$ form a $\mathbb{Z}$-basis of $K(\mo_{c})$. In particular, there exist non-negative integers $\hat{n}_{\tau, \sigma}$ such that  $$[M_{c}(\tau)]=\sum_{\sigma}{\hat{n}_{\tau, \sigma}[L_{c}(\sigma)]}.$$
 Because the grading on all modules is given by the action of $\bf{h}$, these constants satisfy: $\hat{n}_{\tau, \tau}=1$, and for $\tau\ne \sigma$, $\hat{n}_{\tau, \sigma}=0$ unless $h_{c}(\sigma)-h_{c}(\tau)\in \mathbb{Z}_{>0}$. As a consequence, in the appropriate ordering, the matrix of these integers is upper triangular with $1$ on the diagonal, and hence invertible. So, there exist $n_{\tau,\sigma}\in \mathbb{Z}$ such that $$[L_{c}(\tau)]=\sum_{\sigma}{n_{\tau, \sigma}[M_{c}(\sigma)]}.$$

The task of this paper is to find $n_{\tau,\sigma}$. This gives a lot of information about $\mo_{c}$. For example, from it we can easily compute the characters of $L_{c}(\tau)$, as $$\chr_{\LL{c}{\tau}}(t,w)=\sum_{\sigma} n_{\tau,\sigma} \frac{\chi_{\sigma}(w)  t^{h_{c}(\sigma)} }{\operatorname{det}_{\h^{*}}(1-tw)} .$$ The numbers $\hat{n}_{\tau,\sigma}$ are easy to compute from  $n_{\tau,\sigma}$, so this gives the irreducible modules that appear in the decomposition series of every $M_{c}(\tau)$ as well. One thing that Grothendieck group expressions, and consequently characters, don't give, is characterization of extensions: namely, they don't differentiate between split and non-split short exact sequences.

 \subsection{The Group $G_{12}$} Our focus will be on the complex reflection group labeled $G_{12}$ in the Shephard-Todd classification \cite{ST}. It is a group of order $48$, given by generators and relations (see \cite{BMR}) as 
 $$\gtw=\left\langle e,f,g \mid e^{2}=f^{2}=g^{2}=1, (efg)^4=(fge)^4=(gef)^4 \right\rangle .$$
Alternatively, it can be realized as $\GL{2}(\mathbb{F}_{3})$ or as a nonsplit central extension of $\mf{S}_{4}$ by $\mz/2\mz$. More precisely, there is a short exact sequence of groups 
$$1\to \mathbb{Z}_{2} \to G_{12} \to \mf{S}_{4} \to 1$$ with the map $\mathbb{Z}_{2} \to G_{12}$ given by $-1\mapsto (efg)^4$ and the map $G_{12}\to  \mf{S}_{4}$ given by $e \mapsto (12), f\mapsto (34), g\mapsto (23)$. Its reflection representation $\mathfrak{h}$ is two dimensional; if $\zeta$ is a primitive eighth root of unity, the reflection action is given by 
$$e\mapsto \frac{1}{2}\left( \begin{array}{cc} \zeta^3 -\zeta & -\zeta^3 +\zeta \\ -\zeta^3 +\zeta & -\zeta^3 +\zeta \end{array} \right) \qquad
f\mapsto \frac{1}{2}\begin{pmatrix} \zeta^3 -\zeta & \zeta^3 -\zeta \\ \zeta^3 -\zeta & -\zeta^3 +\zeta \end{pmatrix}  \qquad
g\mapsto \begin{pmatrix} 0 & -\zeta \\ \zeta^3 & 0 \end{pmatrix}.$$

The representatives of conjugacy classes in $G_{12}$ are $\{id, (efg)^4, e, eg, ef, fg, efg, egf \}$. It contains $12$ reflections, all in the same conjugacy class with a representative $e$. The character table of $G_{12}$ is given in table \ref{groupcharacter}.

\begin{table}[h!]
\begin{center}
\begin{tabular}{| l | *{8}{c} |} 
\hline 
%& 1 & $(pq)^4$ & $p$ & $q^2$ & $(pq)^2$ & $q$ & $pq$ & $(pq)^5$ \\
& $1$ & $(efg)^4$ & $e$ & $eg$ & $ef$ & $fg$ & $efg$ & $egf$ \\
\hline  
size	& 1 & 1 & 12 & 8 & 6 & 8 & 6 & 6 \\ 
order & 1 & 2 & 2 & 3 & 4 & 6 & 8 & 8 \\ 
\hline 
\hline 
$\rp{1}{+}$ & 1 & 1 & 1 & 1 & 1 & 1 & 1 & 1 \\ 
$\rp{1}{-}$ & 1 & 1 & -1 & 1 & 1 & 1 & -1 & -1 \\ 
$\rp{2}{\;}$ & 2 & 2 & 0 & -1 & 2 & -1 & 0 & 0 \\
$\rp{2}{+}$ & 2 & -2 & 0 & -1 & 0 & 1 & $\sqrt{-2}$ & $-\sqrt{-2}$ \\
$\rp{2}{-}$ & 2 & -2 & 0 & -1 & 0 & 1 & $-\sqrt{-2}$ & $\sqrt{-2}$ \\
$\rp{3}{+}$ & 3 & 3 & 1 & 0 & -1 & 0 & -1 & -1 \\
$\rp{3}{-}$ & 3 & 3 & -1 & 0 & -1 & 0 & 1 & 1 \\
$\rp{4}{\;}$ & 4 & -4 & 0 & 1 & 0 & -1 & 0 & 0 \\
\hline 
\end{tabular}
\end{center}
\caption{Character table of $\gtw$}\label{groupcharacter} 
\end{table}

As stated in the character table, $G_{12}$ has two one-dimensional representations: the trivial one we call $\bf{1}_{+}$ and the signum one we call $\bf{1}_{-}$. The reflection representation is written as $\mathfrak{h}\cong\bf{2}_{+}$ and its dual is $\mathfrak{h}^*\cong \bf{2}_{-}$. Projection to  $\mf{S}_4$ gives, in addition to both one dimensional representations, another three representations $\bf{2}$, $\bf{3}_{+}$ and $\bf{3}_{-}$. Finally, there is a four dimensional representation that can be realized as $\bf{4}\cong \bf{2} \otimes \bf{2}_{+}\cong \bf{2} \otimes \bf{2}_{-}$. It should be noted that the names of the representations are chosen to encode the dimension and the result of tensoring with the signum representation; for example, $\bf{2}_{-}\otimes \bf{1}_{-}\cong \bf{2}_{+}$ and $\bf{2}\otimes \bf{1}_{-}\cong \bf{2}$.

\section{Results}

In this section we list the main results, which describe the structure of category $\mo_{c}$ for an arbitrary complex parameter $c$.  

\begin{thm} If $c$ is not of the form $c=m/12, \, m \in \mathbb{Z},\,  m \equiv 1,3,4,5,6,7,8,9,11 (\operatorname{mod} 12)$, then $\mo_{c}(G_{12},\mathfrak{h})$ is semisimple and $M_{c}(\tau)=L_{c}(\tau)$ for all $\tau$.\label{semisimple}
\end{thm}  

Proof of this theorem is deferred until section \ref{Hecke}.

 From now on we will write all fractions $r/d$ reduced, and assume $r,d>0$. The following result describes the structure of $\mo_{\pm r/d}$ for pairs $(r,d)$ that the previous theorem doesn't cover. Whenever $c$ is clear from the context, we write $L(\tau), M(\tau)$ instead of $[L_{c}(\tau)]$, $[M_{c}(\tau)]$ for readability.
 
\begin{thm}
 \label{main} 
 \begin{enumerate}
 \item If $$[\LL{r/d}{\tau}]=\sum_{\sigma} n_{\tau,\sigma}[\M{r/d}{\sigma}],$$ then $$[\LL{-r/d}{\rp{1}{-} \otimes \tau}]=\sum_{\sigma}{n_{\tau,\sigma}}[\M{-r/d}{\rp{1}{-} \otimes \sigma}] .$$ In particular, $\operatorname{dim}(\LL{r/d}{\tau}) < \infty$ iff $\operatorname{dim}(\LL{-r/d}{\rp{1}{-} \otimes \tau}) < \infty $, and the character formulas  for $c=-r/d$ are easily derived from those for $c=r/d$. \label{main1}

\item The expressions for $[\LL{r/d}{\tau}]$ for $r>0$ and $d \in \left\{2,3,4,12\right\}$ are given below. We include characters for those $\LL{r/d}{\tau}$ that are finite dimensional. For all pairs $(c,\tau)$ that don't appear, $M_{c}(\tau)=L_{c}(\tau)$. 

\begin{itemize}
\item $\mathbf{d=12, r \equiv 1,11,17,19 (\operatorname{mod}24)}$
\begin{eqnarray*}
L(\mathbf{1} _{+})&=&M(\mathbf{1}_{+}) - M(\mathbf{2}_{-}) + M(\mathbf{1}_{-}) \\
L(\mathbf{2}_{-})&=&M(\mathbf{2}_{-}) - M(\mathbf{1}_{-})
\end{eqnarray*}

$L_{r/12}(\mathbf{1}_{+})$ has dimension $r^2$, and character $$\chr_{\LL{r/12}{\rp{1}{+}}}(t,w)=\frac{\operatorname{det}_{\h^*}(1-wt^r)}{\operatorname{det}_{\h^*}(1-wt)} \cdot \ch{1}{+}(w)\cdot t^{1-r} \; .$$ 

\item $\mathbf{d=12, r \equiv 5,7,13,23 (\operatorname{mod}24)}$
\begin{eqnarray*}
L(\mathbf{1} _{+})&=&M(\mathbf{1}_{+}) - M(\mathbf{2}_{+}) + M(\mathbf{1}_{-}) \\
L(\mathbf{2}_{+})&=&M(\mathbf{2}_{+}) - M(\mathbf{1}_{-})
\end{eqnarray*}

$L_{r/12}(\mathbf{1}_{+})$ has dimension $r^2$, and character $$\chr_{\LL{r/12}{\rp{1}{+}}}(t,w)=\frac{\operatorname{det}_{\h}(1-wt^r)}{\operatorname{det}_{\h^*}(1-wt)} \cdot \ch{1}{+}(w)\cdot t^{1-r} \; .$$ 

\item $\mathbf{d=4, r \equiv 1,3 (\operatorname{mod}8)}$
\begin{eqnarray*}
L(\mathbf{1} _{+})&=&M(\mathbf{1}_{+}) - M(\mathbf{3}_{+}) + M(\mathbf{2}_{+}) \\
L(\mathbf{2} _{+})&=&M(\mathbf{2}_{+}) - M(\mathbf{3}_{-}) + M(\mathbf{1}_{-}) \\
L(\mathbf{3} _{+})&=&M(\mathbf{3}_{+}) - M(\mathbf{2}_{+})- M(\mathbf{4}) + M(\mathbf{3}_{-}) \\
L(\mathbf{3} _{-})&=&M(\mathbf{3}_{-}) - M(\mathbf{1}_{-}) \\
L(\mathbf{4})&=&M(\mathbf{4}) - M(\mathbf{3}_{-}) 
\end{eqnarray*}

$\LL{r/4}{\rp{1}{+}}$ has dimension $3r^2$, and character $$\chr_{\LL{r/4}{\rp{1}{+}}}(t,w)=\frac{\operatorname{det}_{\h^*}(1-wt^r)}{\operatorname{det}_{\h^*}(1-wt)} \cdot (\ch{1}{+}(w)\cdot t^{1-3r}+\ch{2}{-}(w)\cdot t^{1-2r}) \; .$$

$\LL{r/4}{\rp{2}{+}}$ has dimension $3r^2$, and character $$\chr_{\LL{r/4}{\rp{2}{+}}}(t,w)=\frac{\operatorname{det}_{\h^*}(1-wt^r)}{\operatorname{det}_{\h^*}(1-wt)} \cdot (\ch{2}{+}(w)\cdot t+\ch{1}{+}(w)\cdot t^{1+r}) \; .$$

$\LL{r/4}{\rp{3}{+}}$ has dimension $3r^2$, and character $$\chr_{\LL{r/4}{\rp{3}{+}}}(t,w)=\frac{\operatorname{det}_{\h^*}(1-wt^r)}{\operatorname{det}_{\h^*}(1-wt)} \cdot \ch{3}{+}(w)\cdot t^{1-r} \; .$$

\item $\mathbf{d=4, r \equiv 5,7 (\operatorname{mod}8)}$
\begin{eqnarray*}
L(\mathbf{1} _{+})&=&M(\mathbf{1}_{+}) - M(\mathbf{3}_{+}) + M(\mathbf{2}_{-}) \\
L(\mathbf{2} _{-})&=&M(\mathbf{2}_{-}) - M(\mathbf{3}_{-}) + M(\mathbf{1}_{-}) \\
L(\mathbf{3} _{+})&=&M(\mathbf{3}_{+}) - M(\mathbf{2}_{-})- M(\mathbf{4}) + M(\mathbf{3}_{-}) \\
L(\mathbf{3} _{-})&=&M(\mathbf{3}_{-}) - M(\mathbf{1}_{-}) \\
L(\mathbf{4})&=&M(\mathbf{4}) - M(\mathbf{3}_{-}) 
\end{eqnarray*}

$\LL{r/4}{\rp{1}{+}}$ has dimension $3r^2$, and character $$\chr_{\LL{r/4}{\rp{1}{+}}}(t,w)=\frac{\operatorname{det}_{\h}(1-wt^r)}{\operatorname{det}_{\h^*}(1-wt)} \cdot (\ch{1}{+}(w)\cdot t^{1-3r}+\ch{2}{+}(w)\cdot t^{1-2r}) \; .$$

$\LL{r/4}{\rp{2}{-}}$ has dimension $3r^2$, and character $$\chr_{\LL{r/4}{\rp{2}{-}}}(t,w)=\frac{\operatorname{det}_{\h}(1-wt^r)}{\operatorname{det}_{\h^*}(1-wt)} \cdot (\ch{2}{-}(w)\cdot t+\ch{1}{+}(w)\cdot t^{1+r}) \; .$$

$\LL{r/4}{\rp{3}{+}}$ has dimension $3r^2$, and character $$\chr_{\LL{r/4}{\rp{3}{+}}}(t,w)=\frac{\operatorname{det}_{\h}(1-wt^r)}{\operatorname{det}_{\h^*}(1-wt)} \cdot \ch{3}{+}(w)\cdot t^{1-r} \; .$$

\item $\mathbf{d=3, r \equiv 1,2 (\operatorname{mod}3)}$
\begin{eqnarray*}
L(\mathbf{1} _{+})&=&M(\mathbf{1}_{+}) - M(\mathbf{2})+M(\mathbf{1}_{-}) \\
L(\mathbf{2})&=&M(\mathbf{2}) - M(\mathbf{1}_{-}) 
\end{eqnarray*}

$\LL{r/3}{\rp{1}{+}}$ has dimension $16r^2$, and character 
$$\chr_{\LL{r/3}{\rp{1}{+}}}(t,w)=\frac{\operatorname{det}_{\h^*}(1-wt^r)}{\operatorname{det}_{\h^*}(1-wt)} \cdot (\ch{1}{+}(w)\cdot t^{1-4r}+\ch{2}{-}(w)\cdot t^{1-3r}+\ch{3}{+}(w)\cdot t^{1-2r}+$$ $$+\ch{4}{\;}(w)\cdot t^{1-r}+\ch{3}{+}(w)\cdot t+\ch{2}{+}(w)\cdot t^{1+r}+\ch{1}{+}(w)\cdot t^{1+2r}) \; .$$

\item $\mathbf{d=2, r \equiv 1 (\operatorname{mod}2)}$
\begin{eqnarray*}
L(\mathbf{1} _{+})&=&M(\mathbf{1}_{+}) - M(\mathbf{3}_{+})+M(\mathbf{2}) \\
L(\mathbf{2})&=&M(\mathbf{2}) - M(\mathbf{3}_{-})+M(\mathbf{1}_{-}) \\
L(\mathbf{3}_{+})&=&M(\mathbf{3}_{+}) - M(\mathbf{2})-M(\mathbf{1}_{-}) \\
L(\mathbf{3}_{-})&=&M(\mathbf{3}_{-}) -M(\mathbf{1}_{-}) \\
\end{eqnarray*}

$\LL{r/2}{\rp{1}{+}}$ has dimension $12r^2$, and character $$\chr_{\LL{r/2}{\rp{1}{+}}}(t,w)=\frac{\operatorname{det}_{\h^*}(1-wt^r)}{\operatorname{det}_{\h^*}(1-wt)} \cdot (\ch{1}{+}(w)\cdot t^{1-6r} + \ch{2}{-}(w)\cdot t^{1-5r} +$$ $$+ \ch{3}{+}(w)\cdot t^{1-4r} + \ch{4}{\;}(w)\cdot t^{1-3r} + \ch{2}{\;}(w)\cdot t^{1-2r}) \; .$$

$\LL{r/2}{\rp{2}{\;}}$ has dimension $12r^2$, and character $$\chr_{\LL{r/2}{\rp{2}{\;}}}(t,w)=\frac{\operatorname{det}_{\h^*}(1-wt^r)}{\operatorname{det}_{\h^*}(1-wt)} \cdot (\ch{2}{\;}(w)\cdot t + \ch{4}{\;}(w)\cdot t^{1-r} + \ch{3}{+}(w)\cdot t^{1-2r} + \ch{2}{+}(w)\cdot t^{1-3r} + \ch{1}{+}(w)\cdot t^{1-4r}) \; .$$

\end{itemize}
\end{enumerate}
\end{thm}

This theorem follows directly from sections \ref{equivalences} and \ref{calculations}.

Using this result, the expressions in the Grothendieck group for $[\M{r/d}{\tau}]$ in terms of $[\LL{r/d}{\sigma}]$ are immediate. We include them for the sake of thoroughness.

\begin{cor}
\label{korolar} 
If $$[\M{r/d}{\tau}]=\sum_{\sigma}{\hat{n}_{\tau,\sigma}}[\LL{r/d}{\sigma}] ,$$ then $$[\M{-r/d}{\rp{1}{-} \otimes \tau}]=\sum_{\sigma}{\hat{n}_{\tau,\sigma}}[\LL{-r/d}{\rp{1}{-} \otimes \sigma}] \; .$$ The Grothendieck group expressions for standard modules in terms of irreducible modules are given below for all $c>0$. For pairs $(c,\tau)$ that don't appear on the list, $M_{c}(\tau)=L_{c}(\tau)$.

\begin{itemize}
\item $\mathbf{d=12, r \equiv 1,11,17,19 (\operatorname{mod}24)}$
\begin{eqnarray*}
M(\mathbf{1} _{+})&=&L(\mathbf{1}_{+}) + L(\mathbf{2}_{-}) \\
M(\mathbf{2}_{-})&=&L(\mathbf{2}_{-}) + L(\mathbf{1}_{-})
\end{eqnarray*}

\item $\mathbf{d=12, r \equiv 5,7,13,23 (\operatorname{mod}24)}$
\begin{eqnarray*}
M(\mathbf{1} _{+})&=&L(\mathbf{1}_{+}) + L(\mathbf{2}_{+}) \\
M(\mathbf{2}_{+})&=&L(\mathbf{2}_{+}) + L(\mathbf{1}_{-})
\end{eqnarray*}

\item $\mathbf{d=4, r \equiv 1,3 (\operatorname{mod}8)}$
\begin{eqnarray*}
M(\mathbf{1} _{+})&=&L(\mathbf{1}_{+}) + L(\mathbf{3}_{+}) + L(\mathbf{4}) \\
M(\mathbf{2} _{+})&=&L(\mathbf{2}_{+}) + L(\mathbf{3}_{-})  \\
M(\mathbf{3} _{+})&=&L(\mathbf{3}_{+}) +L(\mathbf{2}_{+})+L(\mathbf{4}) +L(\mathbf{3}_{-}) \\
M(\mathbf{3} _{-})&=&L(\mathbf{3}_{-}) + L(\mathbf{1}_{-}) \\
M(\mathbf{4})&=&L(\mathbf{4}) +L(\mathbf{3}_{-}) + L(\mathbf{1}_{-})
\end{eqnarray*}

\item $\mathbf{d=4, r \equiv 5,7 (\operatorname{mod}8)}$
\begin{eqnarray*}
M(\mathbf{1} _{+})&=&L(\mathbf{1}_{+}) +L(\mathbf{3}_{+}) + L(\mathbf{4}) \\
M(\mathbf{2} _{-})&=&L(\mathbf{2}_{-}) +L(\mathbf{3}_{-}) \\
M(\mathbf{3} _{+})&=&L(\mathbf{3}_{+}) +L(\mathbf{2}_{-})+ L(\mathbf{4}) + M(\mathbf{3}_{-}) \\
M(\mathbf{3} _{-})&=&L(\mathbf{3}_{-}) + L(\mathbf{1}_{-}) \\
M(\mathbf{4})&=&L(\mathbf{4}) + L(\mathbf{3}_{-}) + L(\mathbf{1}_{-})
\end{eqnarray*}

\item $\mathbf{d=3, r \equiv 1,2 (\operatorname{mod}3)}$
\begin{eqnarray*}
M(\mathbf{1} _{+})&=&L(\mathbf{1}_{+}) +L(\mathbf{2}) \\
M(\mathbf{2})&=&L(\mathbf{2}) +L(\mathbf{1}_{-}) 
\end{eqnarray*}

\item $\mathbf{d=2, r \equiv 1 (\operatorname{mod}2)}$
\begin{eqnarray*}
M(\mathbf{1} _{+})&=&L(\mathbf{1}_{+}) +L(\mathbf{3}_{+})+L(\mathbf{1}_{-}) \\
M(\mathbf{2})&=&L(\mathbf{2}) +L(\mathbf{3}_{-}) \\
M(\mathbf{3}_{+})&=&L(\mathbf{3}_{+}) +L(\mathbf{2})+L(\mathbf{3}_{-})+L(\mathbf{1}_{-}) \\
M(\mathbf{3}_{-})&=&L(\mathbf{3}_{-}) +L(\mathbf{1}_{-}) \\
\end{eqnarray*}

\end{itemize}
\end{cor}

\begin{proof}
Follows directly from the previous theorem.
\end{proof}

Finally, it is a direct calculation to derive the set of aspherical values for $G_{12}$ from this data. A value of $c$ is called \It{aspherical} if there exists a module in category $\mo_{c}$ which has no nontrivial $W$-invariants. This is clearly equivalent to existence of an irreducible module in category $\mo$ which contains no nontrivial $W$-invariants. The description of such values for any complex reflection group is an interesting open question. 

\begin{cor} 
\label{aspherical} The set of aspherical values of for $G_{12}$ is $$\Sigma(G_{12},\h)=\left\{\frac{1}{4}, \frac{-1}{2}, \frac{-1}{3}, \frac{-2}{3},  \frac{-1}{4}, \frac{-3}{4}, \frac{-5}{4}, \frac{-1}{12}, \frac{-5}{12}, \frac{-7}{12}, \frac{-11}{12}, \right\} .$$ \end{cor}

Proof of this corollary is at the end of the section \ref{equivalences}.

\section{Hecke algebras, KZ functors, and the proof of Theorem \ref{semisimple}} 
\label{Hecke}

\subsection{Hecke Algebras} One can associate a Hecke algebra to any complex reflection group, and use known things about representations of this Hecke algebra to study representations of the rational Cherednik algebra associated to the same complex reflection group. We give the basic facts here and refer the reader to \cite{BMR}, \cite{C} and \cite{GP} for details.

 Each reflection $s \in S$ fixes a hyperplane $H_{s} \subset \h$. Let $\h_{\operatorname{reg}}= \h - \bigcup_{s \in S}H_{s}$ be the set of regular points of $\mathfrak{h}$. The complex reflection group $W$ acts on this space, and we call the fundamental group of the quotient $BW:=\pi_{1}(\h_{\operatorname{reg}}/W)$ \textit{the braid group of $W$}. In the case of $G_{12}$, the braid group has a presentation by generators and relations as
 $$BG_{12}=\left<T_{e},T_{f},T_{g}\, | \,(T_{e}T_{f}T_{g})^4=(T_{g}T_{e}T_{f})^4=(T_{f}T_{g}T_{e})^4   \right>,$$ with $T_{e}$ corresponding to the class of a loop around the plane fixed by $e$. 
 
 The Hecke algebra $\mathcal{H}_{q}$ of $W$ is a quotient of the group algebra of the braid group of $W$. In the case of $G_{12}$ it is generated over the ring $\mathbb{C}[q,q^{-1}]$ by generators $T_{e},T_{f},T_{g}$ with relations
$$(T_{e}T_{f}T_{g})^4=(T_{g}T_{e}T_{f})^4=(T_{f}T_{g}T_{e})^4$$
$$(T_{e}-1)(T_{e}+q)=0,\, \, \, (T_{f}-1)(T_{f}+q)=0, \, \, \, (T_{g}-1)(T_{g}+q)=0.$$
 
%Throughout we will take $q=e^{u}$ for a formal parameter $u$. 

\subsection{Semisimplicity and Schur Elements}
It is shown in \cite{BMR} and \cite{C} that $\mathcal{H}_q(W)$ is a free module over $\mc[q,q^{-1}]$ of rank $|W|$ and that as a consequence, we can find a linear map $t: \mathcal{H}_q(W) \rar \mc[q,q^{-1}]$ that is a symmetrizing form for $\mathcal{H}_q(W)$ (ie. a symmetric form inducing an isomorphism $\mathcal{H}_q(W)\to \Hom_{\mc[q,q^{-1}]}(\mathcal{H}_q(W),\mc[q,q^{-1}])$ by $a\mapsto\left(b\mapsto t(ab) \right)$). This map can be written explicitly as follows. The field of definition of $\mathfrak{h}$ as a $W$-representation is $\mathbb{Q}[\sqrt{-2}]$, and thus contains just $2$ roots of unity: $1$ and $-1$. It is then explained in \cite{C} that by replacing the variable $q$ by its appropriate root, in our case $v$ such that $v^2=q$, one gets an algebra $\mc(v)\mathcal{H}_{v^2}(W)$ which is semisimple. Evaluation $v\mapsto 1$, sending the Hecke algebra to the group algebra $\mc[W]$, then induces a bijection between irreducible characters $\chi_{\sigma}$ of $W$ and irreducible characters $\chi_{\sigma}'$ of $\mc(v)\mathcal{H}_{v^2}(W)$. There also exist Schur elements ${s_{\sigma}}\in \mc[v,v^{-1}]$, such that the symmetrizing form $t$ can be expressed as
$$t=\sum_{\sigma \in \operatorname{Irred}(W)}\frac{\chi'_{\sigma}}{s_{\sigma}}.$$

These elements are useful in studying the algebra $\mathcal{H}_q(W)$ for special numerical values of $q$ due to the following lemma, which can be found in \cite{GP}, Section 7.4.   

\begin{lemma} A specialization of $\mathcal{H}_{q}(W)$ at $q=v^2 \in \mc$ is semisimple iff $s_{\sigma}(v) \neq 0$ $\forall \sigma \in \operatorname{Irred}(W)$. In this case there exists a bijective correspondence between $\operatorname{Irred}(\mathcal{H}_{q}(W))$ and $\operatorname{Irred}(W)$. \label{Hqsemisimple}
\end{lemma}

Schur polynomials for specific complex reflection groups have been determined, and can be computed using the CHEVIE algebra package (see \cite{M}). For $G_{12}$, they are given in table \ref{Schur} of the appendix.

%The second statement, which is known as Tits' deformation theorem, tells us in particular that $|\operatorname{Irred}(\mathcal{H}_{q}(W))|=|\operatorname{Irred}(W)|$. This fact will be important.   

\subsection{KZ Functor} There is a functor that enables us to transfer information about representations of the Hecke algebra $\mathcal{H}_{q}(W)$ over $\mathbb{C}$ for numerical $q$ and category $\mo_{c}$ for the rational Cherednik algebra $H_{c}(W,\mathfrak{h})$, for $q=e^{2\pi i c}$. It is the $KZ_c$ functor described below. Again, we give just the information we need and instead refer the reader to \cite{GGOR}. 

 %Since $W$ is a finite group, we can fix a positive definite $W$-invariant Hermitian form on $\h$. For a reflection hyperplane $H_{s}$, choose a vector $v_{s}$ normal to $H_{s}$ with respect to this form. $\gamma_{s} \in \h^{*}$ is the corresponding covector, annihilating $H_{s}$ in $S\h^{*}$. Define $\delta=\prod_{s \in S}{\gamma_{s}}$ and write $H_{\operatorname{reg}}$ for the localization of $H_{c}(W,\h)$ at $\left\{\delta^{i} | i \in \mz_{\geq 0} \right\}$. For $M$ in $\mo_{c}$, we let $M_{\operatorname{reg}}=H_{\operatorname{reg}} \otimes_{H_{c}(W,\h)} M $ be the corresponding localization. 

%In [GGOR], we learn that $\text{KZ}_{c}$ maps $\mo_{c}$ to $B$-$\operatorname{Mod}$, and that for parameters of our form, this map factors through the category $\mathcal{H}_{q}(W)$-$\operatorname{Mod}$. It can be shown that the resulting functor, which by abuse of notation will again be called the $\text{KZ}_{c}$ functor, is exact with kernel given by the subcategory $\mo_{c}^{\operatorname{tor}}$ of modules $M$ such that $M_{\operatorname{reg}}=0$. As noted in the previous section, the Hecke algebra $\mathcal{H}_{q}(W)$ has rank $|W|$ as a $\mc[[q]]$-module. Using this fact, it is found that the induced map $\overline{\text{KZ}}_{c}: \mo_{c} / \mo_{c}^{\operatorname{tor}} \rar \mathcal{H}_{q}(W) \text{-} \operatorname{Mod}$ gives an equivalence of categories.  

Monodromy of a certain localization  $M_{\operatorname{reg}}$ of modules $M$ in category $\mo_{c}$ gives rise to the $KZ$ functor $KZ_{c}:\mo_{c}\to BW \text{-} \operatorname{Mod}$. It is then shown that monodromy factorizes through $H_{q}(W)$ for $q=e^{2\pi i c}$, so it induces a functor (also called KZ by abuse of notation) $KZ_{c}:\mo_{c}\to H_{q}(W) \text{-} \operatorname{Mod}$. Finally, if we denote by $\mo_{c}^{\operatorname{tor}}$ the full subacategory of $\mo_{c}$ consisting of modules $M$ such that $M_{\operatorname{reg}}=0$, the induced functor $KZ_{c}: \mo_{c} / \mo_{c}^{\operatorname{tor}} \to \mathcal{H}_{q}(W) \text{-} \operatorname{Mod}$ is an equivalence of categories.

This functor is used in the following lemma.

\begin{lemma} Let $q=e^{2\pi i c}$. The $\mathbb{C}$ algebra $\mathcal{H}_{q}(W)$ is semisimple if and only if $\mo_{c}$ is semisimple. \label{H=O}
\end{lemma}

\begin{proof} Consider the composition of the quotient and $KZ$ functor: $$\mo_{c}\longrightarrow \mo_{c}/\mo_{c}^{\operatorname{tor}} \longrightarrow \mathcal{H}_{q}(W) \text{-} \operatorname{Mod}.$$ Using the fact that the second arrow is an equivalence of categories and that the simple objects in $\mo$ are $L_{c}(\tau)$ parametrized by $\tau \in \operatorname{Irred}(W)$, we get that always $$|\operatorname{Irred}(W)|=|\operatorname{Irred}(\mo_{c})|\ge |\operatorname{Irred}(\mathcal{H}_{q}(W))|.$$

($\Rightarrow$) If $\mathcal{H}_{q}(W)$ is semisimple, then by lemma \ref{Hqsemisimple} $|\operatorname{Irred}(W)|= |\operatorname{Irred}(\mathcal{H}_{q}(W))|$, so $\mo_{c}^{\operatorname{tor}}=0$, and both functors in the diagram above are equivalences of categories.  Semisimplicity is preserved under equivalence of categories, so $\mo_{c}$ is semisimple.

%Recall that the simple objects of $\mo_{c}$ are of the form $\LL{c}{\tau}$ for $\tau \in \operatorname{Irred}(W)$. By lemma 4.4.1, we know that $|\operatorname{Irred}(\mathcal{H}_{q}(W))|=|\operatorname{Irred}(W)|$. So $\mo_{c}$ and $\operatorname{Irred}(\mathcal{H}_{q}(W))$ have the same number of simple objects. Since $\overline{\text{KZ}}_{c}$ induces an equivalence $\mo_{c} / \mo_{c}^{\operatorname{tor}} \rightarrow \mathcal{H}_{q}(W) \text{-} \operatorname{Mod}$, it follows that $\mo_{c}^{\operatorname{tor}}=0$. Since semisimplicity is preserved under equivalence of categories, we have the result.  

 ($\Leftarrow$) If $\mo_{c}$ is semisimple, then $\M{c}{\tau}=\LL{c}{\tau}$ for all $\tau$. By construction each $\M{c}{\tau}$ is torsion free over $S\h^{*}$, so $\mo_{c}^{\operatorname{tor}}=0$. Again, both functors in the diagram above are equivalences of categories, and $\mathcal{H}_{q}(W)$ is semisimple. 
% Since every non-zero module in $\mo_{c}$ contains an irreducible module, we find that $\mo_{c}^{\operatorname{tor}}=0$. We conclude that $\mathcal{H}_{q}(W) \text{-} \operatorname{Mod}$ is a semisimple category, and therefore that $\mathcal{H}_{q}(W)$ is a semisimple algebra. \Qed         
\end{proof}

\subsection{Proof of Theorem 3.1} Combining lemma \ref{Hqsemisimple} and lemma \ref{H=O}, we conclude that $\mo_{c}$ is semisimple iff $s_{\sigma}(e^{2 \pi i c}) \neq 0$ $\forall \sigma \in \operatorname{Irred}(W)$. By inspection of the Schur elements for $G_{12}$, which can be found  in table \ref{Schur} in the appendix, the result follows immediately.

\section{Equivalences of categories and proof of Corollary \ref{aspherical}}
\label{equivalences}

In this section, we describe three equivalences of categories $\mo_{c}$ for various $c$. As a consequence, we will reduce the number of parameters $c$ for which we study category $\mo_{c}$ to a small finite number, study it in these cases, and then use equivalences of categories to describe $\mo_{c}$ in all other cases. 

We note that by theorem \ref{semisimple}, only parameters of the form $c=m/12$ where $m \not \equiv 0,\pm2 (\operatorname{mod} 12)$ have a nontrivial category $\mo_{c}$.

\subsection{Reduction to $c>0$ and the proof of theorem \ref{main} (1)} 
\label{cpositive}
There exists an isomorphism $$H_{c}(W,\h) \rar H_{-c}(W,\h)$$ defined on generators to be the identity on $\h$ and $\h^*$, and to send $w$ to $\rp{1}{-}(w) \cdot w$. Twisting by this isomorphism is an equivalence of categories $\mo_c$ and $\mo_{-c}$. This equivalence exchanges $\M{c}{\tau}$ and $\M{-c}{\rp{1}{-} \otimes \tau}$, and therefore also exchanges their irreducible quotients $\LL{c}{\tau}$ and $\LL{-c}{\rp{1}{-} \otimes \tau}$. Tensoring by $\rp{1}{-}$ determines a permutation of the characters of $W$. Therefore, if we know the expression for $\LL{c}{\tau}$ in $K(\mo_{c})$, the expression for $\LL{-c}{\tau}$ is the same with $+/-$ exchanged. In particular, we easily see that $\LL{c}{\tau}$ is finite dimensional iff $\LL{-c}{\bf{1}_{-}\otimes \tau}$ is finite dimensional. This proves the first part of theorem \ref{main}.

\subsection{Scaling functors and reduction to $c=1/d$ if $d=3,4,12$} There exists an equivalence of categories $$\Phi_{1/d,r/d}:\mo_{1/d}\to \mo_{r/d}$$ for $d=3,4,12$ and $r$ a positive integer relatively prime to $d$ (see \cite{R} theorem 5.12.). In particular, there exists a permutation $\varphi_{1/d,r/d}$ of $\mathrm{Irred}(W)$ such that $$\Phi_{1/d,r/d}(M_{1/d}(\tau))=M_{r/d}(\varphi_{1/d,r/d}(\tau))$$  $$\Phi_{1/d,r/d}(L_{1/d}(\tau))=L_{r/d}(\varphi_{1/d,r/d}(\tau)).$$ We calculate the permutation $\varphi$ explicitly for the case of $G_{12}$ in section \ref{permute}. This enables us to transfer character formulas in $\mo_{1/d}$, which we will calculate in section \ref{calculations}, to character formulas in category $\mo_{r/d}$.

\subsection{Shift functors and reduction to $c=1/d$ if $d=2$}
\label{shift}
Scaling functors described above are known to be equivalences of categories for $d\ne 2$ and  conjectured to be equivalences for $d=2$. In absence of the proof of that, we use shift functors for equivalences between various parameters of the form $c=r/2$. For references on shift functors, see \cite{BEG}.

Denote $H_{c}=H_{c}(W,\mathfrak{h})$, for $c$ a constant. Let $$e_{+}=\frac{1}{|W|}\sum_{w\in W} w\in \mathbb{C}W,\, \, \, \, e_{-}=\frac{1}{|W|}\sum_{w\in W} \mathbf{1}_{-}(w)w\in \mathbb{C}W$$ be projections to $W$-invariants and anti-invariants, respectively. The paper \cite{BEG} then shows there is an isomorphism of filtered algebras $$\phi_{c}:e_{+}H_{c}e_{+}\to e_{-}H_{c+1}e_{-}.$$ It induces the natural equivalence between categories of their representations $$\Phi_{c}:\mathcal{O}(e_{+}H_{c}e_{+})\to \mathcal{O}(e_{-}H_{c+1}e_{-}).$$

Next, for $\varepsilon\in \{ +,-\}$, one defines functors $$F_{c}^{\varepsilon}:\mathcal{O}_{c}\to \mathcal{O}_{c}(e_{\varepsilon}H_{c}e_{\varepsilon}) \qquad G_{c}^{\varepsilon}:\mathcal{O}_{c}\to \mathcal{O}_{c}(e_{\varepsilon}H_{c}e_{\varepsilon})$$ by $$F_{c}^{\varepsilon}(V)=e_{\varepsilon}V \qquad G_{c}^{\varepsilon}(Y)=H_{c}e_{\varepsilon}\otimes_{e_{\varepsilon}H_{c}e_{\varepsilon}}Y.$$

Denote by $\mathcal{O}_{c}^{\varepsilon}$ the full subcategory of $\mathcal{O}_{c}$ consisting of modules such that $e_{\varepsilon}V=0$. This is a Serre subcategory so the quotients  $\mathcal{O}_{c}/\mathcal{O}_{c}^{\varepsilon}$ make sense. Using the fact that for $V$ a simple module in $\mathcal{O}_{c}$, either $e_{\varepsilon}V=$ or $H_{c}e_{\varepsilon}V=V$, it is easy to see that $F_{c}^{\varepsilon}$ and $G_{c}^{\varepsilon}$, now understood as functors to and from quotient categories, are mutually inverse equivalences of categories between $\mathcal{O}_{c}/\mathcal{O}_{c}^{\varepsilon}$ and $\mathcal{O}(e_{\varepsilon}H_{c}e_{\varepsilon})$. So, we have the following diagram of equivalences of categories:
$$\begin{CD} 
\mathcal{O}_{c}/\mathcal{O}_{c}^+ @>F_{c}^{+}>> \mathcal{O}(e_{+}H_{c}e_{+}) @>\Phi_{c}>> \mathcal{O}(e_{-}H_{c+1}e_{-})  @>G_{c+1}^{-}>> \mathcal{O}_{c+1}/\mathcal{O}_{c+1}^-
\end{CD} $$

Consider the compositions of these functors $S_{c}^+=G_{c+1}^{-}\circ \Phi_{c}\circ F_{c}^+$, and $S_{c}^-=G_{c}^{+}\circ \Phi_{c}^{-1} \circ F_{c+1}^-$. These are mutually inverse equivalences of categories between $\mathcal{O}_{c}/\mathcal{O}_{c}^+$ and $\mathcal{O}_{c+1}/\mathcal{O}_{c+1}^-$. Moreover, if $\mathcal{O}_{c}^+=0$, then $S_{c}^+$ is an equivalence of categories between $\mathcal{O}_{c}$ and $\mathcal{O}_{c+1}$; because $\mathcal{O}_{c}$ and $\mathcal{O}_{c+1}$ have the same number of simple objects (equal to the number of irreducible representations of $W$), and if $\mathcal{O}_{c+1}^-\ne 0$, then $\mathcal{O}_{c+1}/\mathcal{O}_{c+1}^-$ has strictly less.

When $c$ is a half integer, call the composition functor $\Phi_{c,c+1}=G_{c+1}^{-}\circ \Phi_{c}\circ F_{c}^{+}$. Again, there exists a permutation $\varphi_{c,c+1}$ of irreducible representations of $W$ such that $$\Phi_{c,c+1}(M_{c}(\tau))=M_{c+1}(\varphi_{c,c+1}(\tau))$$  $$\Phi_{c,c+1}(L_{c}(\tau))=L_{c+1}(\varphi_{c,c+1}(\tau)).$$ We can conclude that 
\begin{enumerate}
\item If all modules in $\mo_{c}$ contain a nontrivial $W$-invariant, in other words if $c$ is spherical, then $\Phi_{c,c+1}$ is an equivalence of categories $\mo_{c}\to \mo_{c+1}.$
\item In that case, every module in $\mo_{c+1}$ contains a nontrivial $W$-anti-invariant.
\item From formulas for relating characters of irreducible representations, given in section \ref{permute}, it will follow that if $c>0$ is spherical then so is $c+1$.
\end{enumerate}

Combining (1) and (3), the algorithm for computing character formulas in $\mo_{r/2}$ is now the following: we compute them consecutively for $c=1/2, 3/2, 5/2, \ldots$ until we find a spherical $c$. From that value on, all $\Phi_{c,c+1}, \Phi_{c+1,c+2},\ldots $ are equivalnces of categories, so we can use formulas from \ref{permute} to compute characters for all larger half integers. 

In case of $G_{12}$, it will turn out that $c=1/2$ is already spherical, so $\mo_{r/2}$ is equivalent to $\mo_{1/2}$ for all positive odd $r$.

%For the case $d=2$, [R 5.12] is conjectural. In absence of a proof, we must resort to different methods. Let $e_{+}=\frac{1}{|W|}\sum_{w \in W}{w} \in \mc[W]$ be the projection to $W$-invariants. From [BC], we learn that if $e_{+}\LL{c}{\tau} \neq 0$, then there exists an equivalence of categories $\mo_{c} \rar \mo_{1+c}$. We note that $e_{+}\LL{c}{\tau} \neq 0$ iff $\LL{c}{\tau}$ contains a $W$-subrepresentation isomorphic to $\rp{1}{+}$. So to obtain our result it suffices to show that $\LL{1/2}{\tau}$ seen as a representation of $W$ contains a copy of $\rp{1}{+}$ for each $\tau \in \operatorname{Irred}(W)$. 
 
 %The parabolic subgroups of $W$ for nonzero points are isomorphic to $0$ or $\mz/2\mz$. We determine the representation theory of $H_c(0,\h_0)$ and $H_c(\mz/2\mz,\h_{\mz/2\mz})$, and find that the irreducible modules of both algebras contain $\rp{1}{+}$. Using [BE 3.1,3.2,4.1], we can then conclude that all infinite dimensional $\LL{c}{\tau}$ contain a copy of the trivial representation. 

 %In section 6.2.4, we will find that $\LL{1/2}{\rp{2}{\;}}$ and $\LL{1/2}{\rp{1}{+}}$ are finite dimensional, while the rest are infinite dimensional. Computing characters, we see that both contain $\rp{1}{+}$. Therefore we will be able to deduce the structure of $\mo_{r/2}$ from that of $\mo_{1/2}$ by repeated use of a shift functor. We will again have to check whether certain switches need to be made.   

\subsection{ Permutation $\varphi_{c,c'}$} 
\label{permute}

Let $d \in \left\{2,3,4,12 \right\}$, $r>0$ integer relatively prime to $d$, and assume that $\Phi_{1/d,r/d}$ described above (a scaling functor in case $d\ne 2$ or a composition of shift functors in case $d=2$) is an equivalence of categories (As explained above, this assumption is automatically satisfied for scaling functors, and satisfied in the case of $G_{12}$ for shift functors). We want to calculate the permutation $\varphi_{1/d,r/d}$ such that  $$\Phi_{1/d,r/d}(M_{1/d}(\tau))=M_{r/d}(\varphi_{1/d,r/d}(\tau))$$  $$\Phi_{1/d,r/d}(L_{1/d}(\tau))=L_{r/d}(\varphi_{1/d,r/d}(\tau)).$$

An effective way of calculating this permutation is given in \cite{GG} by formula (11). Namely, pick 
$g \in \operatorname{Gal}(\mq(e^{2 \pi i /2d})/\mq)$ such that $g(e^{2 \pi i /d})=e^{2 \pi i r/d}$ and let $\eta= e^{-2 \pi i r/2d} \cdot g(e^{2 \pi i /2d})$. Then the irreducible characters $\chi'$ of the Hecke algebra $\mathcal{H}_{q}(W)$ over $\mathbb{C}[q,q^{-1}]$ are related by 
$$\chi'_{\varphi_{1/d,r/d}(\tau)}(w)(q)=(g\chi'_{\tau}(w))(\eta q).$$
After evaluating at $q=1$ to get irreducible characters of $W$, this becomes:
\begin{itemize}
\item If $\eta=1$, then $\chi'_{\varphi_{1/d,r/d}(\tau)}=g\chi'_{\tau}$, so $\varphi_{1/d,r/d}$ is a permutation given by the action of $g$ on the characters of $W$. In case of $G_{12}$, as all characters are rational except characters of $\mathbf{2}_{-}$ and $\mathbf{2}_{+}$ which contain $\sqrt{-2}$ (see table \ref{groupcharacter}), this will include checking if  $\sqrt{-2}\in \mq(e^{2 \pi i /2d})$ and what the effect of $g$ on it is. 
\item If $\eta=-1$, then the character formula becomes  $\chi'_{\varphi_{1/d,r/d}(\tau)}(w)(1)=g\chi'_{\tau}(w)(-1)$. The left hand side can be interpreted as a group character, while the right hand side needs to be transformed. It can be shown that in the case of $G_{12}$, $\chi_{\sigma}(w)(-1)=\chi_{\sigma}(w)(1)$ if $\sigma \ne \mathbf{2}_{-}, \mathbf{2}_{+}$, and  $\chi_{\mathbf{2}_{-}}(w)(-1)=\chi_{\mathbf{2}_{+}}(w)(1), \, \chi_{\mathbf{2}_{+}}(w)(-1)=\chi_{\mathbf{2}_{-}}(w)(1)$. So, in the case $\eta= -1$, the permutation $\varphi_{1/d,r/d}$ is the composition of the permutation given by the action of $g$ on the characters of $W$ and the transposition of $\mathbf{2}_{+}$ and $\mathbf{2}_{-}$.
\end{itemize}
While there might be several choices for $g$, the permutation $\phi$ doesn't depend on them.

Let us illustrate the computation of  $\varphi_{1/d,r/d}$ in the case of $d=4$. Let $\zeta=e^{2 \pi i/8}$ and note $\sqrt{-2}=\zeta+\zeta^3 \in \mq(e^{2 \pi i /8})$. It suffices to determine permutations for $r \equiv \pm 1,\pm 3 (\operatorname{mod} 8)$.
\begin{itemize}
\item $r \equiv 1$. In this case, we can choose $g=id$, so $\eta=1$ and $\varphi=id$.
\item $r \equiv 3$. We must choose $g \in \operatorname{Gal}(\mq(\zeta)/\mq)$ such that $g(\zeta^2)=\zeta^6$. There are two choices, $\zeta \mapsto \zeta^3$ and $\zeta \mapsto \zeta^7$. For the sake of example, let us consider both. In the first case, $\eta=\zeta^{-3}g(\zeta)=1$ and $g(\zeta+\zeta^3)=\zeta+\zeta^3$. In the second case, $\eta=\zeta^{-3}g(\zeta)=-1$ and $g(\zeta+\zeta^3)=\zeta^{-1}+\zeta^{-3}=-\sqrt{-2}$. So, in either case we get $\varphi=id$. 
\item $r \equiv -3$. In this case, we must find a $g$ such that $g(\zeta^2)=\zeta^{10}=\zeta^2$. Choose $g=id$. Then $\eta=\zeta^{-5}g(\zeta)=-1$, so $\varphi=(\rp{2}{+} \rp{2}{-})$ is a transposition. 
\item $r \equiv -1$. Here we need $g$ such that $g(\zeta^2)=\zeta^{14}=\zeta^6$. We can choose the element corresponding to $\zeta \mapsto \zeta^3$. It fixes $\zeta+\zeta^3$, implies $\zeta^{-7}g(\zeta)=-1$ and finally $\varphi=(\rp{2}{+} \rp{2}{-})$.
\end{itemize}

%Therefore in the case of $d=4$, we conclude that the standard modules corresponding to $\rp{2}{+}$ and $\rp{2}{-}$ switch for $r \equiv -1,-3 (\operatorname{mod} 8)$, and are unchanged for $r \equiv 1,3 (\operatorname{mod} 8)$. Proceeding in this way, we can determine the permuations for the remaining cases with the exception of $d=3$. This case however is unimportant. We will see in section 6.2.2 that $\M{1/3}{\rp{2}{+}}$ and $\M{1/3}{\rp{2}{-}}$ are simple and live in different blocks, so permuting them does not affect their descriptions in $K(\mo_{1/3})$. 

Let us summarize these results in a proposition: 

\begin{proposition} {For $d \in \left\{2,3,4,12\right\}$ and $(r,d)=1$, $r>0$, there exist an equivalence of categories $$\Phi_{1/d,r/d}: \mo_{1/d} \rar \mo_{r/d}$$ and a permutation $\vphi_{1/d,r/d}$ of $\operatorname{Irred}(W)$ such that $$\Phi(\LL{1/d}{\tau}) = \LL{r/d}{\vphi_{1/d,r/d}(\tau)}$$ $$\Phi(\M{1/d}{\tau}) = \M{r/d}{\vphi_{1/d,r/d}(\tau)} .$$} The permutation $\vphi_{1/d,r/d}$ is as follows:
\begin{itemize}
\item If $d=2$, then $\vphi=\operatorname{id}$;
\item If $d=3,\, r\equiv 1,2 (\operatorname{mod} 6)$, then $\vphi=\operatorname{id}$;
\item If $d=3,\, r\equiv 4,5 (\operatorname{mod} 6)$, then $\vphi=(\rp{2}{+} \rp{2}{-})$;
\item If $d=4,\, r \equiv 1,3 (\operatorname{mod} 8)$ then  $\vphi=\operatorname{id}$;
\item If $d=4,\, r \equiv 5,7 (\operatorname{mod} 8)$ then $\vphi=(\rp{2}{+} \rp{2}{-})$;
\item If $d=12,\, r \equiv 1,11,17,19 (\operatorname{mod} 24)$ then  $\vphi=\operatorname{id}$;
\item If $d=12,\, r \equiv 5,7,13,23 (\operatorname{mod} 24)$ then $\vphi=(\rp{2}{+} \rp{2}{-})$;
\end{itemize}
\end{proposition}
 
 As a consequence we get formulas for transforming characters of irreducible modules. More precisely, if $$[\LL{1/d}{\tau}]=\sum_{\sigma}{n_{\tau,\sigma}}[\M{1/d}{\sigma}],$$ then $$[\LL{r/d}{\vphi_{1/d,r/d}(\tau)}]=\sum_{\sigma}{n_{\tau,\sigma}}[\M{r/d}{\vphi_{1/d,r/d}(\sigma)}].$$  An easy consequence of these formulas and the formula for the character of $M_{c}(\tau)$ is the following corollary:
 
\begin{cor} Let $\gamma_{1/d,r/d} \in \operatorname{Gal}(\mq(\sqrt{-2}))$ be the identity or complex conjugation, depending on whether $\vphi_{1/d,r/d}$ is identity or transposition of $\rp{2}{+}$ and $\rp{2}{-}$. Then $$\chr_{\LL{r/d}{\tau}}(t,w)=\frac{\operatorname{det}_{\varphi_{1/d,r/d}(\h^*)}(1-wt^r)}{\operatorname{det}_{\h^*}(1-wt)} \cdot t^{1-r} \cdot \gamma_{1/d,r/d}(\chr_{\LL{1/d}{\tau}}(t^r,w)).$$  In particular, $\LL{1/d}{\tau}$ is finite dimensional iff $\LL{r/d}{\vphi_{1/d,r/d}(\tau)}$ is, with $\operatorname{dim}(\LL{r/d}{\vphi_{1/d,r/d}(\tau)})=r^2 \cdot \operatorname{dim}(\LL{1/d}{\tau})$. 
\end{cor}

Note that property (3) from the end of section \ref{shift} follows from this corollary. 

\subsection{Proof of corollary \ref{aspherical}} Assuming theorem \ref{main}, we are trying to compute the set of all aspherical values of $c$ for $G_{12}$, meaning values of $c$ for which there are modules in $\mo_{c}$ with no nontrivial $W$-invariants. It is obviously enough to just inspect the irreducible modules for this property. First of all, according to \cite{BE} 4.1, a module in $\mo_{c}(W,\mathfrak{h})$ which has no nontrivial $W$-invariants is either finite dimensional or induced from a module in $\mo_{c}(W',\mathfrak{h'})$, where $W'$ is a parabolic subgroup of $W$. It is easy to see that the only nontrivial parabolic subgroups of $G_{12}$ are isomorphic to $\mathbb{Z}_{2}$, and that the only aspherical value for $\mathbb{Z}_{2}$ is $c=-1/2$. Indeed, the modules $L_{-1/2}(\mathbf{1}_{-}), L_{-1/2}(\mathbf{2}), L_{-1/2}(\mathbf{3}_{-})$ have no $W$-invariants, and $-1/2$ is aspherical for $G_{12}$. 

We are left with the task of examining the finite dimensional modules for invariants. First assume $c>0$. It is clear from the character formulas for the finite dimensional modules from \ref{main} that all the $L_{r/d}(\tau)$, for $0<r/d<1$, contain an invariant, except $L_{r/4}(\bf{3}_{+})$. Checking values $r/4$ one by one we see that $c=1/4$ is indeed aspherical (with $L_{1/4}(\bf{3}_{+})$ containing no invariant), $c=3/4$ and $c=5/4$ are spherical ($L_{3/4}(\bf{3}_{+})$ and $L_{3/4}(\bf{3}_{+})$ contain $\bf{1}_{+}\subset S^{2}\mathfrak{h}^*\otimes \bf{3}_{+}$. Finally, from the properties of shift functors at the end of section \ref{shift} it follows that if $c$ is aspherical, then every module in $\mo_{c+1}$ contains both a $W$-invariant and a $W$-antiinvariant, so in particular $c+1$ is spherical too, and conclude that the only positive aspherical value for $G_{12}$ is $1/4$.

Next, it is clear from \ref{cpositive} that for $c<0$ the module $L_{c}(\tau)$ has no $W$-invariants if and only if the corresponding module $L_{-c}(\mathbf{1}_{-}\otimes \tau)$ has no $W$-antiinvariants. A similar computation as above, case by case, gives us the list in corollary \ref{aspherical}.

\section{Background}

 In this section, we review several concepts that will provide the basis for the methods we use in section \ref{calculations} in the proof of the main result \ref{main}.    
 
 \subsection{Contravariant Form $B$} 
We explained above that the standard module $M_{c}(\tau)$ has a unique maximal proper submodule $J_{c}(\tau)$ and that we are interested in studying the irreducible quotient $L_{c}(\tau)=M_{c}(\tau)/J_{c}(\tau)$. It turns out there is an efficient way to recognize vectors belonging to $J_{c}(\tau)$.

Consider a rational Cherednik algebra $H_{c}(W,\mathfrak{h})$ and a standard module $M_{c}(\tau)$. Let $\bar{c}:S\to \mathbb{C}$ be another conjugation invariant function, given by  $\bar{c}(s) =c(s^{-1})$, and consider the rational Cherednik algebra $H_{\bar{c}}(W,\h^*)$ and a standard module $\M{\bar{c}}{\tau^*}$ for it. Using the universal mapping property one can show that there is a unique, up to multiplication by a nonzero scalar, bilinear form $$B: \M{c}{\tau} \times \M{\bar{c}}{\tau^*} \to \mc$$ which satisfies:
\begin{eqnarray*}
B(xu,v)=B(u,xv)  && x \in \h^* \\
 B(yu,v)=B(u,yv) && y \in \h \\ 
 B(wu,v)=B(u,w^{-1}v) && w \in W
\end{eqnarray*}

The usefulness of this form is in the fact that its radical in the first coordinate, defined as the set of all vectors $v\in M_{c}(\tau)$ such that $B(v,-)=0$, is exactly equal to $J_{c}(\tau)$. This can be seen from the universal mapping property and the definition of $B$. Moreover, $B$ can be efficiently inductively computed using the following rules:
\begin{itemize}
\item for $i\ne j$, graded pieces $S^{i}\mathfrak{h^*}\otimes \tau \subset \M{c}{\tau}$ and $S^{j}\mathfrak{h}\otimes \tau^* \subset \M{\bar{c}}{\tau^*}$ are orthogonal with respect to $B$;
\item on $\tau \otimes \tau^* \subset \M{c}{\tau}\otimes  \M{\bar{c}}{\tau^*}$, the form $B$ is the standard pairing of a vector space with its dual;
\item once the form $B_{n}: \left( S^{n}\h^{*} \otimes \tau \right) \otimes \left( S^{n}\h \otimes \tau^{*}\right)  \rar \mc$ has been computed, the form $B_{n+1}: \left( S^{n+1}\h^{*} \otimes \tau \right) \otimes \left( S^{n+1}\h \otimes \tau^{*}\right)  \rar \mc$  is given by $B_{n+1}(u,yv)=B_{n}(D_{y}(u),v)$.
\end{itemize}
 
This gives an efficient way of programming the form $B$ in one of the algebra software packages. Then calculating the rank of the matrix of $B$ restricted to some particular graded piece of $M_{c}(\tau)$ gives us the dimension of that graded piece of $L_{c}(\tau)$. We do these computations in MAGMA algebra software \cite{Magma}.

\subsection{Lowest $\mathbf{h}$-weights} \label{lowestparagraph}
All modules in category $\mo_{c}$ have an inner grading by the action of $\mathbf{h}$. If the Grothendieck group expression for a standard module in terms of irreducible modules is $$[M_{c}(\tau)]=\sum_{\sigma} \hat{n}_{\tau,\sigma}[L_{c}(\sigma)]$$ for some nonnegative integers $\hat{n}_{\tau,\sigma}$, then the action of $\mathbf{h}$ implies that $\hat{n}_{\tau,\tau}=1$ and that the only $\sigma\ne \tau$ for which $\hat{n}_{\tau,\sigma}$ is nonzero satisfy $h_{c}(\sigma)-h_{c}(\tau)\in \mathbb{Z}_{>0}$. Consequently, the matrix of integers $\hat{n}_{\tau,\sigma}$ falls apart into block diagonal matrices such that $\sigma$ and $\tau$ label the row and column in the same block if and only if $h_{c}(\sigma)-h_{c}(\tau)\in \mathbb{Z}$. Moreover, with the appropriate ordering of $\sigma$ in the same block (the ordering corresponding to $h_{c}(\sigma)-h_{c}(\tau)$), these blocks are themselves upper triangular matrices with $1$ on diagonal. Hence, inverting this matrix $[\hat{n}_{\tau,\sigma}]$ to get the matrix of integers $[n_{\tau,\sigma}]$ that allow us to express irreducible modules in terms of standard ones, $$[L_{c}(\tau)]=\sum_{\sigma} n_{\tau,\sigma}[M_{c}(\sigma)],$$ we conclude that the matrix $[n_{\tau,\sigma}]$ is again of the same form. In other words, $n_{\tau,\tau}=1$, and the only $\sigma\ne \tau$ for which $n_{\tau,\sigma}$ is nonzero satisfy $h_{c}(\sigma)-h_{c}(\tau)\in \mathbb{Z}_{>0}$. Because of that, values of $h_{c}(\tau)$ for all $\tau$ carry a lot of information, and we will use this information in section \ref{calculations}. For $W=G_{12}$, they are calculated in table \ref{lowest}.

\begin{table}[!h]
\begin{center}
\begin{tabular}{| l | *{8}{c} |} 
\hline 
$\tau$	& $\rp{1}{+}$ & $\rp{1}{-}$ & $\rp{2}{\;}$ & $\rp{2}{+}$ & $\rp{2}{-}$ & $\rp{3}{+}$ & $\rp{3}{-}$ & $\rp{4}{\;}$ \\ 
\hline
\hline
$\sum_{s \in S}{s} \; |_{\tau}$ & 12 & -12 & 0 & 0 & 0 & 4 & -4 & 0 \\
$h_{c}(\tau)$ & 1-12c & 1+12c & 1 & 1 & 1 & 1-4c & 1+4c & 1 \\
\hline
\end{tabular}
\end{center}
\caption{Lowest $\mathbf{h}$-weights}\label{lowest}
\end{table}

\subsection{A Useful Lemma} 
The following result can be found in \cite{ES} as Lemma 3.5.

\begin{lemma}Let $\sigma \subset \h^{*} \otimes \tau \subset \M{c}{\tau}$ be an irreducible subrepresentation of $W$. The elements of $\h$ act on $\sigma$ by zero iff $h_{c}(\sigma)-h_{c}(\tau)=1$. \label{35}
\end{lemma}

 %\textit{Proof.} We know that the action of the basis element $y_{i}$ on $\M{c}{\tau}$ is given by the Dunkl operator $D_{y_{i}}$. On $S^{1}\h^{*} \otimes \tau $ this action determines a linear map $\tau \otimes \h^{*} \otimes \h \rar \tau$ or equivalently an operator $\tau \otimes \h^{*} \rar \tau \otimes \h^{*}$. This operator can be easily determined and is given by the formula $$1- \sum_{s \in S}{\frac{2 c(s)}{1-\lambda_{s}}(1-s) \otimes s} \; .$$ So on $\sigma$ this endomorphism is the scalar $1+h_{c}(\tau)-h_{c}(\sigma)$. This gives the result. \Qed 

We know from the previous subsection that $n_{\tau,\sigma}$ is nonzero only if $h_{c}(\sigma)-h_{c}(\tau)\in \mathbb{Z}_{>0}$. This lemma tells us that if $h_{c}(\sigma)-h_{c}(\tau)=1$ and $\sigma $ is in the appropriate graded piece of $M_{c}(\tau)$, the converse is also true. We will use this fact repeatedly in section \ref{calculations}.

\subsection{A-matrix} \label{A-matrix} The following lemma appears in  \cite{S}, and provides a simple yet computationally very effective necessary condition on the structure constants $n_{\tau, \sigma}$ for the module $L_{c}(\tau)$ to be finite dimensional.

Fix $\tau$ and $c$ and consider the representation $\LL{c}{\tau}$ is finite dimensional. Write $[\LL{c}{\tau}]=\sum_{\sigma}{n_{\tau,\sigma}}[\M{c}{\sigma}]$. For $g\in W$, consider $\operatorname{det}_{\h^{*}}(1-gt)$ as a polynomial in $t$, and let $t_{g,1}, t_{g,2}$ be its roots. Clearly, $g\mapsto \{ t_{g,1}, t_{g,2}\}$ is constant on conjugacy classes.

Define a matrix $A$, with columns indexed by $\operatorname{Irred}(W)$ and rows indexed by the ordered pairs $(g,i)$, for $g$ a conjugacy class in $W$ and $i=1,2$, by setting the element in the row labeled by $(g,i)$ and column labeled by $\sigma$ to be equal to $t_{g,i}^{h_{c}(\sigma)} \chi_{\sigma}(g)$. 

\begin{lemma}If $\LL{c}{\tau}$ is finite dimensional, then column vector $[n_{\tau,\sigma}]_{\sigma \in \mathrm{Irred}{W}}$ is a nullvector or $A$. \label{useA} \end{lemma}

\begin{proof} Since $\LL{c}{\tau}$ is finite dimensional, it follows that its character is a polynomial. On the other hand, it can be written as $$\chr_{\LL{c}{\tau}}(g,t)=\frac{\sum_{\sigma}{n_{\tau,\sigma} \cdot \chi_{\sigma}(g) \cdot t^{h_{c}(\sigma)}}}{\operatorname{det}_{\h^{*}}(1-gt)} .$$ For every $t_{g,i}$, the denominator vanishes; since the character doesn't have a pole, the numerator must vanish as well. This gives the desired claim. \end{proof} 
 
The matrix $A$ is easy to compute and gives a strong condition on $n_{\tau,\sigma}$. We compute it using MAGMA, for variable $c$. The matrix and basis vectors for its null-spaces at $c=1/12,1/4,1/3,1/2$ are given in tables \ref{matrixA} and \ref{nulls} the appendix.

 \subsection{Induction and Restriction Functors} 
 
A subgroup $W'$ of a complex reflection group $W$ is called parabolic if there exists $a \in \h$ such that $W'$ is the stabilizer of $a$. It can be shown that a parabolic subgroup is itself a complex reflection group, with reflection representation $\mathfrak{h}'=\h/\h^{W'}$, where $\h^{W'}$ is the subspace of $W'$-invariant elements. For example, the only nontrivial parabolic subgroups of $G_{12}$ are isomorphic to $\mathbb{Z}_{2}$ and contain just one reflection and the identity.

Let $c'$ the restriction of $c$ to the reflections in $W'$. Induction and restriction functors, introduced in \cite{BE}, give a way of relating modules for $H_{c}(W,\h)$ and of $H_{c'}(W',\h/\h^{W'})$. We omit the details of their construction and give only the properties we will use. 

\begin{proposition}\label{ind}There exist induction and restriction functors $$\operatorname{Res}_{a}: \mo_c (W,\mathfrak{h}) \rar \mo_{c'}(W',\mathfrak{h}')$$   $$\operatorname{Ind}_{a}: \mo_{c'}(W',\mathfrak{h}') \rar \mo_c(W,\mathfrak{h})$$ associated to $a \in \h$ such that $W'=\operatorname{Stab}(a)$. These functors are exact. The following formulas hold for generic $c$, and on the level of Grothendieck group for every $c$:
$$\operatorname{Res}_{a}(M_c(W,\h,\tau))=\bigoplus_{\sigma \in \operatorname{Irred}(W')} \operatorname{dim}(\operatorname{Hom}(\sigma,\tau|_{W'}))M_{c'}(W',\h',\sigma),$$ $$\operatorname{Ind}_{a}(M_{c'}(W',\h',\sigma))=\bigoplus_{\tau \in \operatorname{Irred}(W)} \operatorname{dim}(\operatorname{Hom}(\sigma,\tau|_{W'}))M_c(W,\h,\tau).$$
\end{proposition}

\section{Calculations for $c=1/12, 1/3, 1/4, 1/2$}
\label{calculations}
This section finishes the proof of theorem \ref{main} by giving explicit calculations of the structure of category $\mo_{c}$ for $G_{12}$ in the remaining cases $c \in \left\{1/12, 1/3, 1/4, 1/2 \right\}$. After describing $\mo_{c}$ for $c=1/2$, we also need to show $1/2$ is a spherical value, so the shift functor \ref{shift} is an equivalence and we can use the information about $1/2$ to derive description of $\mo_{r/2}$ for all $r\in \mathbb{Z}_{>0}$.

%\subsection{Conditions on Finite Dimensional Representations} Using lemma \ref{useA} and the corresponding table \ref{matrixA}, we find necessary conditions on finite dimensional $\LL{c}{\tau}$ for $c \in \left\{1/12, 1/3, 1/4, 1/2 \right\}$ by evaluating the $A$-matrix and determining basis vectors for it nullspace, which we know to be nontrivial. The results are included in Table 3 of the appendix.  

\subsection{Proof of theorem \ref{main} for $c=1/12$}
Evaluating the expressions in table \ref{lowest} at $c=1/12$, we get the lowest $\mathbf{h}$-weights in table \ref{c12}.

\begin{table}[!h] 
\begin{tabular}{| l | *{8}{c} |} 
\hline 
$\tau$	& $\rp{1}{+}$ & $\rp{1}{-}$ & $\rp{2}{\;}$ & $\rp{2}{+}$ & $\rp{2}{-}$ & $\rp{3}{+}$ & $\rp{3}{-}$ & $\rp{4}{\;}$ \\ 
\hline
\hline
$h_{1/12}(\tau)$ & 0 & 2 & 1 & 1 & 1 & 2/3 & 4/3 & 1 \\
\hline
\end{tabular}
\caption{Lowest $\bf{h}$-weights for $c=1/12$} \label{c12}
\end{table}

There is no $\sigma$ such that $h_{1/12}(\sigma)-h_{1/12}(\rp{1}{-})$ is positive. So, discussion in \ref{lowestparagraph} implies that all the constants $n_{\mathbf{1}_{-},\sigma}$ are $0$ except  $n_{\mathbf{1}_{-},\mathbf{1}_{-}}=1$, in other words that $$\M{1/12}{\rp{1}{-}}=\LL{1/12}{\rp{1}{-}}$$ is simple. Similarly, and there is no  $\sigma$ such that $h_{1/12}(\sigma)-h_{1/12}(\rp{3}{-})$ or $h_{1/12}(\sigma)-h_{1/12}(\rp{3}{+})$ are integers, so the same reasoning implies $$\M{1/12}{\rp{3}{+}}=\LL{1/12}{\rp{3}{+}}$$ $$\M{1/12}{\rp{3}{-}}=\LL{1/12}{\rp{3}{-}}$$ are simple. We will use this argument a lot. 

The only $\sigma$ that can appear in the Grothendieck group expressions for $L_{1/12}(\tau)$ if $\tau=\rp{2},\bf{2}_{+},\bf{2}_{-},\bf{4}$ is $\sigma=\bf{1}_{-}$, again due to the condition that $h_{c}(\sigma)-h_{c}(\tau)\in \mathbb{Z}_{>0}$. In that case, this difference is $h_{c}(\sigma)-h_{c}(\tau)=1$, so the graded piece $S^{1}\mathfrak{h}^*\otimes \tau\subseteq M_{1/12}(\tau)$ contains a subrepresentation isomorphic to $\bf{1}_{-}$. We decompose these group representations into irreducible subrepresentations and get 

\begin{table}[!h] 
\begin{tabular}{| l | *{4}{c} |} 
\hline 
$\tau$	&  $\rp{2}{\;}$ & $\rp{2}{+}$ & $\rp{2}{-}$ & $\rp{4}{\;}$ \\ 
\hline
\hline
$S^{1}\h^{*} \otimes \tau $ & $\rp{4}{\;}$ & $\rp{1}{+} \oplus \rp{3}{-}$ & $\rp{1}{-} \oplus \rp{3}{+}$ & $\rp{2}{\;} + \rp{3}{+} \oplus \rp{3}{-} $   \\
\hline
\end{tabular}
\end{table}

So we conclude $$\M{1/12}{\rp{2}{\;}}=\LL{1/12}{\rp{2}{\;}},$$ $$\M{1/12}{\rp{2}{+}}=\LL{1/12}{\rp{2}{+}},$$ $$\M{1/12}{\rp{4}{\;}}=\LL{1/12}{\rp{4}{\;}}$$ are simple because $\rp{1}{-}$ does not appear. However,  $S^{1}\h^* \otimes \rp{2}{-}$ contains a copy of $\rp{1}{-}$. By lemma \ref{35}, which says this necessary condition is in fact sufficient if the difference $h_{c}(\sigma)-h_{c}(\tau)$ is $1$, we conclude that $$ \lb{1/12}{\rp{2}{-}}=\mb{1/12}{\rp{2}{-}} - \mb{1/12}{\rp{1}{-}} \; .$$

Similarly, we calculate the decomposition of the first graded piece of $\M{1/12}{\rp{1}{+}}$. We find that $S^{1}\h^* \otimes \rp{1}{+} = \rp{2}{-}$. %$$S^{2}\h^* \otimes \rp{1}{+} = \rp{3}{+}.$$ 
Again by lemma \ref{35}, the entire graded piece $S^{1}\h^* \otimes \rp{1}{+} = M_{1/12}(\rp{1}{+})[1]$ consists of singular vectors. They generate a representation isomorphic to a quotient of  $\M{1/12}{\rp{2}{-}}$, which we saw was either $\M{1/12}{\rp{2}{-}}$ or $\LL{1/12}{\rp{2}{-}}$. Comparing dimensions of $M_{1/12}(\rp{1}{+})$ and $\M{1/12}{\rp{2}{-}}$ in the $\mathbf{h}$-eigenspace $2$, we see that $\operatorname{dim}(\M{1/12}{\rp{2}{-}}[2])=\operatorname{dim}(\mathfrak{h}^*\otimes \rp{2}{-})=4$, while $\operatorname{dim}(\M{1/12}{\rp{1}{+}}[2])=\operatorname{dim}(S^2\mathfrak{h}^*\otimes \rp{1}{+})=3$, so $M_{1/12}(\rp{2}{-})$ can't be a submodule of $M_{1/12}(\rp{1}{+})$. Dimensions of $L_{1/12}(\rp{2}{-})$ and $M_{1/12}(\rp{1}{+})$ match in all higher degrees, so we conclude $J_{1/12}(\rp{1}{+})=L_{1/12}(\rp{2}{-})$ and
$$\lb{1/12}{\rp{1}{+}}=\mb{1/12}{\rp{1}{+}} - \lb{1/12}{\rp{2}{-}}=\mb{1/12}{\rp{1}{+}} - \mb{1/12}{\rp{2}{-}} + \mb{1/12}{\rp{1}{-}}.$$ 

This module is one dimensional, and is in fact the only finite dimensional irreducible module at $c=1/12$. Its character is easy to compute directly or from the expression in terms of standard modules. We note that its character corresponds vector $e_{1/12}$ in the table \ref{nulls} in the appendix. The part of the statement of theorem \ref{main} for $c=r/12$ now follows from this directly by applying scaling functors.

\subsection{Proof of theorem \ref{main} for $c=1/4$} Using table \ref{lowest}, we get lowest weights for $c=1/4$ in table \ref{c4}. 

\begin{table}[!h] 
\begin{tabular}{| l | *{8}{c} |} 
\hline 
$\tau$	& $\rp{1}{+}$ & $\rp{1}{-}$ & $\rp{2}{\;}$ & $\rp{2}{+}$ & $\rp{2}{-}$ & $\rp{3}{+}$ & $\rp{3}{-}$ & $\rp{4}{\;}$ \\ 
\hline
\hline
$h_{1/12}(\tau)$ & -2 & 4 & 1 & 1 & 1 & 0 & 2 & 1 \\
\hline
\end{tabular}
\caption{Lowest $\bf{h}$-weights for $c=1/4$}\label{c4} 
\end{table}

It immediately follows $$\M{1/4}{\rp{1}{-}}=\LL{1/4}{\rp{1}{-}}.$$ We will need the following decompositions: $$S^{2}\h^{*} \otimes \rp{3}{+}=\rp{1}{+}\oplus \rp{2}{\;} \oplus \rp{3}{+} \oplus \rp{3}{-}, \; \; S^{4}\h^{*} \otimes \rp{3}{+}=\rp{1}{+}\oplus \rp{2}{\;} \oplus 2 \cdot \rp{3}{+}\oplus 2 \cdot \rp{3}{-}$$

\begin{table}[h] 
\begin{center} 
\begin{tabular}{| p{1.5cm} | p{2cm} | p{2cm} | p{2cm} | p{1.5cm} | p{3.9cm} |} 
\hline 
$\tau$ & $\rp{2}{\;}$ & $\rp{2}{+}$ & $\rp{2}{-}$ & $\rp{3}{+}$ & $\rp{4}{\;}$  \\ 
\hline
\hline
$S^{1}\h^{*} \otimes \tau$ &  $\rp{4}{\;}$ & $\rp{1}{+} \oplus \rp{3}{-}$ & $\rp{1}{-}\oplus \rp{3}{+}$ & $\rp{2}{+}\oplus \rp{4}{\;}$ & $\rp{2}{\;}\oplus \rp{3}{+}\oplus \rp{3}{-}$  \\
\hline
$S^{3}\h^{*} \otimes \tau$ &  $\rp{2}{+}\oplus \rp{2}{-}\oplus \rp{4}{\;}$ & $\rp{2}{\;} \oplus \rp{3}{+} \oplus\rp{3}{-}$ & $\rp{2}{\;} \oplus \rp{3}{+} \oplus\rp{3}{-}$ & \; & $\rp{1}{+} \oplus \rp{1}{-} \oplus \rp{2}{\;}\oplus 2 \cdot \rp{3}{+} \oplus 2 \cdot \rp{3}{-}$ \\
\hline 
\end{tabular}
\end{center}
\end{table}

 So $$\M{1/4}{\rp{2}{\;}}=\LL{1/4}{\rp{2}{\;}}$$  $$\M{1/4}{\rp{2}{-}}=\LL{1/4}{\rp{2}{-}}$$ are simple, because $\rp{3}{-}$ and $\rp{1}{+}$  don't appear in the right weight space.

 We will use the expression for $\LL{1/4}{\rp{2}{+}}$ to study $\LL{1/4}{\rp{3}{-}}$. First of all, lemma \ref{35} implies that $\M{1/4}{\rp{2}{+}}$ contains a quotient of $\M{1/4}{\rp{3}{-}}$. Comparing dimensions in $\bf{h}$-weight space $4$, we see that this module cannot be $\M{1/4}{\rp{3}{-}}$. So, $$\lb{1/4}{\rp{3}{-}}=\mb{1/4}{\rp{3}{-}} - \mb{1/4}{\rp{1}{-}}$$ 
 $$\lb{1/4}{\rp{2}{+}}=\mb{1/4}{\rp{2}{+}} - \mb{1/4}{\rp{3}{-}} + \mb{1/4}{\rp{1}{-}}.$$ The latter is finite dimensional.
 
To analyze  $\M{1/4}{\rp{4}{\;}}$, first use lemma \ref{35} to conclude its Grothendieck group expression is $$\lb{1/4}{\rp{4}{\;}}=\mb{1/4}{\rp{4}{\;}} - \mb{1/4}{\rp{3}{-}} +n_{\rp{4}{\;} \bf{1_{-}}}\mb{1/4}{\rp{1}{-}} .$$ To find $n_{\rp{4}{\;},\bf{1_{-}} }$, we calculate the rank of the form $B$ in $\bf{h}$-weight space $4$ in MAGMA. We find it is 7. Since $\operatorname{dim}_{3}(\M{1/4}{\rp{4}{\;}})=16$, and $\operatorname{dim}_{2}(\M{1/4}{\rp{3}{-}})=9$, there exist two possibilities. Either $\M{1/4}{\rp{1}{-}}$ appears as a submodule of $\M{1/4}{\rp{4}{\;}}$ and $\M{1/4}{\rp{3}{-}}$, or as a submodule of neither. While our methods cannot distinguish the two cases, in both we can conclude $$\lb{1/4}{\rp{4}{\;}}=\mb{1/4}{\rp{4}{\;}} - \mb{1/4}{\rp{3}{-}} \; .$$ 

 By lemma \ref{35}, we know that $$\lb{1/4}{\rp{3}{+}}=\mb{1/4}{\rp{3}{+}}-\lb{1/4}{\rp{2}{+}}-\lb{1/4}{\rp{4}{;}} \pm \textrm{modules with lowest weight} >1.$$ Looking at dimensions in $\bf{h}$-weight space $1$, we see that $\LL{1/4}{\rp{3}{+}}$ is finite dimensional. Writing down the condition that its character must be a polynomial, which we do by saying that it needs to be a linear combination of vectors in table \ref{nulls}, we get $$\lb{1/4}{\rp{3}{+}}=\mb{1/4}{\rp{3}{+}} - \mb{1/4}{\rp{2}{+}} - \mb{1/4}{\rp{4}{\;}} + \mb{1/4}{\rp{3}{-}},$$ which corresponds to $-e_{1/4}^3$.
 
 Using MAGMA, we find that the form $B$ restricted to $M_{1/4}(\rp{1}{+})[0]$ is zero. Therefore $\M{1/4}{\rp{1}{+}}$ is finite dimensional. Referring to table \ref{nulls} again, it is easily seen that the only possible linear combination is $e_{1/4}^1+e_{1/4}^3$, giving
$$\lb{1/4}{\rp{1}{+}}=\mb{1/4}{\rp{1}{+}} - \mb{1/4}{\rp{3}{+}} + \mb{1/4}{\rp{2}{+}} \; .$$

\subsection{Proof of theorem \ref{main} for $c=1/3$}
First we calculate the lowest weights $h_{c}(\tau)$, see table \ref{c5}.
\begin{table}[!h] 
\begin{tabular}{| l | *{8}{c} |} 
\hline 
$\tau$	& $\rp{1}{+}$ & $\rp{1}{-}$ & $\rp{2}{\;}$ & $\rp{2}{+}$ & $\rp{2}{-}$ & $\rp{3}{+}$ & $\rp{3}{-}$ & $\rp{4}{\;}$ \\ 
\hline
\hline
$h_{1/3}(\tau)$ & -3 & 5 & 1 & 1 & 1 & -1/3 & 7/3 & 1 \\
\hline
\end{tabular}
\caption{Lowest $\bf{h}$-weights for $c=1/3$}\label{c5}
\end{table}

 Since $h_{1/3}(\rp{1}{-}) > h_{1/3}(\tau)$ for all $\tau \neq \rp{1}{-}$, it follows that $$\M{1/3}{\rp{1}{-}}=\LL{1/3}{\rp{1}{-}}.$$ SImilarly, $h_{1/3}(\rp{3}{\pm}) - h_{1/3}(\tau) \notin \mz$ for all $\tau \neq \rp{3}{\pm}$, so $$\M{1/3}{\rp{3}{+}}=\LL{1/3}{\rp{3}{+}}$$ $$\M{1/3}{\rp{3}{-}}=\LL{1/3}{\rp{3}{-}}.$$

 If $\M{1/3}{\rp{1}{-}}$ is contained in $\M{1/3}{\tau}$ as a submodule for $\tau=\bf{2},\bf{2}_+,\bf{2}_{-},\bf{4}$ so that $h_{1/3}(\tau)=1$, then $\M{1/3}{\tau}$ contains $\bf{1}_{-}$ in the fourth grade. Decomposing $S^{4}\h^{*} \otimes \tau $ in all of these cases, we get:

\begin{table}[!h] 
\begin{tabular}{| l | c|c|c|c| |} 
\hline 
$\tau$	&  $\rp{2}{\;}$ & $\rp{2}{+}$ & $\rp{2}{-}$ & $\rp{4}{\;}$ \\ 
\hline
\hline
$S^{4}\h^{*} \otimes \tau $ & $\rp{1}{+}\oplus \rp{1}{-}\oplus \rp{2}{\;}\oplus \rp{3}{+}\oplus \rp{3}{-}$ & $\rp{2}{-} \oplus  2 \cdot \rp{4}{\;}$ & $\rp{2}{+} \oplus  2 \cdot \rp{4}{\;}$ & $2 \cdot \rp{2}{+}\oplus  2 \cdot \rp{2}{-}\oplus  2 \cdot \rp{4}{\;} $  \\
\hline
\end{tabular}
\end{table}

Hence $$\M{1/3}{\rp{2}{+}}=\LL{1/3}{\rp{2}{+}}$$ $$\M{1/3}{\rp{2}{-}}=\LL{1/3}{\rp{2}{-}}$$ $$\M{1/3}{\rp{4}{\;}}=\LL{1/3}{\rp{4}{\;}}$$ 

We see that $S^{4}\h^* \otimes \rp{2}{\;}$ contains a copy of $\rp{1}{-}$. Let us first analyze $\LL{1/3}{\rp{1}{+}}$ and return to the description of $\LL{1/3}{\bf{2}}$ after that.

$S^{4}\h^* \otimes \rp{1}{+}= \rp{2}{\;} \oplus  \rp{3}{+}$, so it follows that $\M{1/3}{\rp{1}{+}}$ can contain a quotient of $\M{1/3}{\rp{2}{\;}}$. Using MAGMA, we compute that the rank of the contravariant form $B$ on \linebreak[3] $\M{1/3}{\rp{1}{+}}[1]=S^{4}\h^* \otimes \rp{1}{+}$. It is $3$ while the dimension of this graded piece is $5$, So $\M{1/3}{\rp{1}{+}}$ does contain a quotient of $\M{1/3}{\rp{2}{\;}}$. Calculating $\operatorname{dim}(\M{1/3}{\rp{2}{\;}}[5])=10$, while $\operatorname{dim}(\M{1/3}{\rp{1}{+}}[5])=9$. So we conclude that the module  $\M{1/3}{\bf{2}}$ cannot be simple. Combining this with the above analysis we conclude
$$  \lb{1/3}{\rp{2}{\;}}=\mb{1/3}{\rp{2}{\;}} - \mb{1/3}{\rp{1}{-}}$$
$$\lb{1/3}{\rp{1}{+}}=\mb{1/3}{\rp{1}{+}} - \mb{1/3}{\rp{2}{\;}} + \mb{1/3}{\rp{1}{-}}.$$

\subsection{Proof of theorem \ref{main} for $c=1/2$}
 Using the expressions in Table \ref{lowest}, we find the following lowest weights for $c=1/2$.     

\begin{table}[!h] 
\begin{center} 
\begin{tabular}{| l | *{8}{c} |} 
\hline 
$\tau$	& $\rp{1}{+}$ & $\rp{1}{-}$ & $\rp{2}{\;}$ & $\rp{2}{+}$ & $\rp{2}{-}$ & $\rp{3}{+}$ & $\rp{3}{-}$ & $\rp{4}{\;}$ \\ 
\hline
\hline
$h_{1/12}(\tau)$ & -5 & 7 & 1 & 1 & 1 & -1 & 3 & 1 \\
\hline
\end{tabular}
\end{center}
\caption{Lowest $\bf{h}$-weights for $c=1/2$}\label{c2}
\end{table}

It immediately follows that $$\M{1/2}{\rp{1}{-}}=\LL{1/2}{\rp{1}{-}}.$$ We will need the following decompositions:
 $$S^{4}\h^* \otimes \rp{3}{+}=\rp{1}{+}\oplus \rp{2}{\;}\oplus2 \cdot \rp{3}{+}\oplus2 \cdot \rp{3}{-}, \; \; S^{8}\h^* \otimes \rp{3}{+}=\rp{1}{+}\oplus \rp{1}{-}\oplus 2 \cdot \rp{2}{\;}\oplus 4 \cdot \rp{3}{+}\oplus 3 \cdot \rp{3}{-}$$ 

\begin{table}[!h] 
\begin{center} 
\begin{tabular}{| p{1.5cm} | p{2cm} | p{2cm} | p{2cm} | p{1.8cm} | p{2.3cm} | } 
\hline 
$\tau$	&  $\rp{2}{\;}$ & $\rp{2}{+}$ & $\rp{2}{-}$ & $\rp{3}{+}$ & $\rp{4}{\;}$ \\ 
\hline
\hline
$S^{2}\h^{*} \otimes \tau $ &  $\rp{3}{+}\oplus\rp{3}{-}$  &  $\rp{2}{-} \oplus \rp{4}{\;}$  & $\rp{2}{+}\oplus \rp{4}{\;}$ & $\rp{1}{+}\oplus \rp{2}{\;}\oplus \rp{3}{+}\oplus \rp{3}{-}$ & $\rp{2}{+}\oplus \rp{2}{-}\oplus 2 \cdot \rp{4}{\;}$   \\
\hline
$S^{6}\h^{*} \otimes \tau $ &  $\rp{2}{+}\oplus 2 \cdot \rp{3}{+}\oplus 2 \cdot \rp{3}{-}$  &  $2 \cdot \rp{2}{+}\oplus \rp{2}{-} \oplus  2 \cdot \rp{4}{\;}$ & $\rp{2}{+}\oplus 2 \cdot \rp{2}{-} \oplus  2 \cdot \rp{4}{\;}$ & \; &  $2 \cdot \rp{2}{+}\oplus 2 \cdot \rp{2}{-} \oplus  5 \cdot \rp{4}{\;}$   \\
\hline
\end{tabular}
\end{center}
\end{table}

 So $$\M{1/2}{\rp{2}{+}}=\LL{1/2}{\rp{2}{+}}$$ $$\M{1/2}{\rp{2}{-}}=\LL{1/2}{\rp{2}{-}}$$ $$\M{1/2}{\rp{4}{\;}}=\LL{1/2}{\rp{4}{\;}}$$ are simple because $\rp{3}{-}$ and $\rp{1}{-}$ do not appear in the appropriate graded pieces.
 
Completely analogous to previous cases, we calculate the rank of $B$ on the $6$ dimensional space $\M{1/2}{\rp{2}{\;}}[3]$ and get that it is $3$. Comparing dimensions of $\M{1/2}{\rp{2}{\;}}[7]$ and $\M{1/2}{\rp{3}{-}}[7]$ (they are $14$ and $15$), we get 
$$\lb{1/2}{\rp{2}{\;}}=\mb{1/2}{\rp{2}{\;}} - \mb{1/2}{\rp{3}{-}} + \mb{1/2}{\rp{1}{-}} $$ 
$$\lb{1/2}{\rp{3}{-}}=\mb{1/2}{\rp{3}{-}} - \mb{1/2}{\rp{1}{-}} \; .$$ 

% We will use the expression for $\M{1/2}{\rp{2}{\;}}$ to show that $\M{1/2}{\rp{3}{-}}$ does contain $\M{1/2}{\rp{1}{-}}$. From this we will obtain $$\lb{1/2}{\rp{3}{-}}=\mb{1/2}{\rp{3}{-}} - \mb{1/2}{\rp{1}{-}} \; .$$ By inspection of the possible linear combinations of the vectors $e_{1/2}^{1}$ and $e_{1/2}^{2}$ in Table 3, we can then conclude that $\LL{1/2}{\rp{3}{-}}$ is not finite dimensional.

% $\M{1/2}{\rp{2}{\;}}$ can contain a quotient of $\M{1/2}{\rp{3}{+}}$ starting at the second grade, but cannot contain $\M{1/2}{\rp{1}{-}}$. Using MAGMA, we find that the rank of $B_2(\rp{2}{\;})$ is 3 instead of 6. Therefore $\M{1/2}{\rp{2}{\;}}$ contains a quotient of $\M{1/2}{\rp{3}{+}}$. Since $\operatorname{dim}_{6}(\M{1/2}{\rp{2}{\;}})=14$, and $\operatorname{dim}_{4}(\M{1/2}{\rp{3}{-}})=15$, it follows that $\M{1/2}{\rp{3}{-}}$ contains $\M{1/2}{\rp{1}{-}}$. So $\LL{1/2}{\rp{2}{\;}}$ is finite dimensional with $$\lb{1/2}{\rp{2}{\;}}=\mb{1/2}{\rp{2}{\;}} - \mb{1/2}{\rp{3}{-}} + \mb{1/2}{\rp{1}{-}} \; .$$ 

Before considering $\LL{1/2}{\rp{3}{+}}$, let us determine $\LL{1/2}{\rp{1}{+}}$. Using MAGMA, we find that the rank of $B$ on $M_{1/4}(\rp{1}{+})[-1]$ is $2$, while the space is $5$ dimensional. So 
$\M{1/2}{\rp{1}{+}}$ contains a quotient of $\M{1/2}{\rp{3}{+}}$ starting at the $\bf{h}$-space $-1$, and is hence finite dimensional. Since $\operatorname{dim}(\M{1/2}{\rp{1}{+}}[1])=7$, and $\operatorname{dim}(\M{1/2}{\rp{3}{+}}[1])=9$, it follows that $\M{1/2}{\rp{3}{+}}$ is not simple, but contains a set of singular vectors isomorphic to $\rp{2}{\;}$ in $\bf{h}$-space $1$. Using decompositions $S^{2}\h^{*} \otimes \rp{3}{+}=\rp{1}{+} \oplus \rp{2}{\;} \oplus  \rp{3}{+} \oplus  \rp{3}{-}$ and $S^{6}\h^{*} \otimes \rp{1}{+}=\rp{1}{+} \oplus  \rp{3}{+} \oplus  \rp{3}{-}$ and table \ref{nulls}, we obtain $$\lb{1/2}{\rp{1}{+}}=\mb{1/2}{\rp{1}{+}} - \mb{1/2}{\rp{3}{+}} + \mb{1/2}{\rp{2}{\;}}$$ corresponding to $e_{1/2}^{1}$.

 We know that $$[\LL{1/2}{\rp{3}{+}}]=[\M{1/2}{\rp{3}{+}}]-[L_{1/2} (\bf{2} ) ]\pm \textrm{modules with lowest weights $>1$}.$$ 
Using MAGMA, we find that $B$ on $\M{1/2}{\rp{3}{+}}[3]$ has rank $9$ and conclude $$\lb{1/2}{\rp{3}{+}}=\mb{1/2}{\rp{3}{+}}-\mb{1/2}{\rp{2}{\;}}+X \cdot \mb{1/2}{\rp{1}{-}}$$ for some integer $X$.   

 We will determine $X$ using induction functors. Consider a point $a \neq 0$ on a reflection hyperplane. We know that the isotropy group of $a$ is isomorphic to $\mz/2\mz$. Let $\epsilon_+,\epsilon_- \in \operatorname{Irred}(\mz/2\mz)$ be the trivial and sign representations. It is easily seen that the irreducible representations of $H_{1/2}(\mz/2\mz,\epsilon_-)$ have Grothendieck group expressions $$\lb{1/2}{\epsilon_-}=\mb{1/2}{\epsilon_-}$$ $$\lb{1/2}{\epsilon_+}=\mb{1/2}{\epsilon_+}-\mb{1/2}{\epsilon_-} .$$ We now use proposition \ref{ind} to deduce $$[\operatorname{Ind}_{a}(\LL{1/2}{\epsilon_+})]=\mb{1/2}{\rp{1}{+}}-\mb{1/2}{\rp{1}{-}}+\mb{1/2}{\rp{3}{+}}-\mb{1/2}{\rp{3}{-}} \; .$$ Using the expressions already determined, this can be rewritten as $$[\operatorname{Ind}_{a}(\LL{1/2}{\epsilon_-})]=\lb{1/2}{\rp{1}{+}}+\lb{1/2}{\rp{2}{\;}}+2 \cdot \lb{1/2}{\rp{3}{+}} - (2X+2) \cdot \lb{1/2}{\rp{1}{-}} \; .$$ Therefore $-(2X+2)\ge 0$ and $X \leq -1$. 

Looking at the multiplicity of $\rp{1}{-}$ in $S^{8}\h^{*} \otimes \rp{3}{+}$ (it is $1$), in $S^{6}\h^{*} \otimes \rp{2}{\;}$ (it is $0$), it follows that $\rp{1}{-}$ appears $1-0+X\ge 0$ times in $\LL{1/2}{\rp{3}{+}}[7]$. This implies $X=-1$. So  $$\lb{1/2}{\rp{3}{+}}=\mb{1/2}{\rp{3}{+}}-\mb{1/2}{\rp{2}{\;}} - \mb{1/2}{\rp{1}{-}}.$$

Looking at the characters of all $L_{1/2}(\tau)$, we find they all have nontrivial $W$-invariants. Hence, shift functors $\Phi_{1/2,3/2}$, $\Phi_{3/2,5/2}$, etc. are all equivalences and we can use these descriptions of category $\mo_{1/2}$ to deduce descriptions of $\mo_{r/2}$ for all positive $r$.

\section*{Appendix: Computational data used}
\label{app}

\begin{table}[!h]
\begin{center}
%\scalebox{.87}{
$\begin{pmatrix}
1 &  1 &  2 &  2 &  2 &  3 &  3 &  4 \\
1 &  1 &  2 &  2 &  2 &  3 &  3 &  4  \\
-y^{-144} &  -y^{144} &  -2 &  2 &  2 &  -3y^{-48} &  -3y^{48} &  4 \\
-y^{-144} &  -y^{144} &  -2 &  2 &  2 &  -3y^{-48} &  -3y^{48} &  4 \\
-y^{-144} &  y^{144} &  0 &  0 &  0  & -y^{-48} &  y^{48} &  0 \\
1  & -1 &  0 &  0 &  0 &  1 &  -1 &  0 \\
(\xi^{4} - 1)y^{-96}  & (\xi^{4} - 1)y^{96} & -\xi^{4} + 1 &  -\xi^{4} + 1 &   -\xi^{4} + 1 &  0 &  0 &  \xi^{4} - 1 \\
-\xi^{4}y^{-192} &  -\xi^{4}y^{192} &  \xi^{4} &  \xi^{4} &  \xi^{4} &  0 &  0 &  -\xi^{4} \\
-\xi^{6}y^{-216} &  -\xi^{6}y^{216} &  -2\xi^{6} &  0 &  0 &  \xi^{6}y^{-72}  & \xi^{6}y^{72} &  0 \\
\xi^{6}y^{-72} &  \xi^{6}y^{72} &  2\xi^{6} &  0  & 0 &  -\xi^{6}y^{-24} &  -\xi^{6}y^{24} &  0 \\
\xi^{4}y^{-48} &  \xi^{4}y^{48} &  -\xi^{4} &  \xi^{4} &  \xi^{4} &  0 &  0 &  -\xi^{4} \\
(-\xi^{4} + 1)y^{-240} &  (-\xi^{4} + 1)y^{240}  & \xi^{4} - 1  & -\xi^{4} + 1 &  -\xi^{4} + 1 &  0  & 0  & \xi^{4} - 1 \\
(-\xi^{5} + \xi) y^{-252} &  (\xi^{5} - \xi)y^{252} &  0 &  -\xi^{6} - 1 &  \xi^{6} + 1 &  (\xi^{5} - \xi)y^{-84} &  (-\xi^{5} + \xi)y^84  & 0 \\
-\xi^3y^{-180}  & \xi^{3}y^{180} &  0 &  \xi^{6} - 1 &  -\xi^{6} + 1 &  \xi^{3}y^{-60}  & -\xi^{3}y^{60} &  0 \\
(\xi^{5} - \xi)y^{-108} &  (-\xi^{5} + \xi)y^{108} &  0 &  -\xi^{6} - 1 &  \xi^{6} + 1 &  (-\xi^{5} + \xi)y^{-36} &  (\xi^{5} - \xi)y^{36} &  0  \\
\xi^{3}y^{-36}  & -\xi^{3}y^{36} &  0  & \xi^{6} - 1 &  -\xi^{6} + 1  & -\xi^{3}y^{-12}  & \xi^{3}y^{12} &  0 
\end{pmatrix}$ %}
\end{center}
\caption{$A$-matrix for parameter $c$ where $\xi =e^{2 \pi i/24}$ and $y=\xi^{c}$. The columns are indexed in the order $\rp{1}{+}$, $\rp{1}{-}$, $\rp{2}{\;}$, $\rp{2}{+}$, $\rp{2}{-}$, $\rp{3}{+}$, $\rp{3}{-}$, $\rp{4}{\;}$.}\label{matrixA}
\end{table}

\begin{table}[!h]
\begin{center}
%\scalebox{.8}{
\begin{tabular}{| p{1.5cm} || p{10cm} |}
\hline 
$c=1/12$ &  $ e_{1/12}=(1,1,0,0,-1,0,0,0)  $ \\ 
\smallskip $c=1/4$ & \smallskip $ e_{1/4}^{1}=(1,0,0,0,0,0,1,-1) $, $ e_{1/4}^{2}=(0,1,0,0,0,1,0,-1) $, $ e_{1/4}^{3}=(0,0,0,1,0,-1,-1,1) $ \\ 
\smallskip $c=1/3$ & \smallskip $ e_{1/3}=(1,1,-1,0,0,0,0,0) $ \\ 
\smallskip $c=1/2$ & \smallskip $ e_{1/2}^{1}=(1,0,1,0,0,-1,0,0)$, $ e_{1/2}^{2}=(0,1,1,0,0,0,-1,0)$ \\
\hline 
\end{tabular} %}
\end{center}
\caption{Nullspace bases for $A$-matrix at $c=1/12,1/4,1/3,1/2$} \label{nulls}
\end{table}

\begin{table}[!h] 
\begin{center}
%\scalebox{.8}{
\begin{tabular}{| p{1cm} || p{14cm} |}
\hline 
$s_{\rp{1}{+}}$ & $v^{-24}(v - \xi)(v + \xi)(v - \xi^3)^2(v + \xi^3)^2(v - \xi^4)(v + \xi^4)(v - \xi^5)(v + \xi^5)(v - \xi^6)^2(v + \xi^6)^2(v - \xi^7)(v + \xi^7)(v - \xi^8)(v + \xi^8)(v - \xi^9 )^2(v + \xi^9)^2(v - \xi^{11})(v + \xi^{11})$ \\ 

\smallskip $s_{\rp{1}{-}}$ & \smallskip $(v - \xi)(v + \xi)(v - \xi^3)^2(v + \xi^3)^2(v - \xi^4)(v + \xi^4)(v - \xi^5)(v + \xi^5)(v - \xi^6)^2(v + \xi^6)^2(v - \xi^7)(v + \xi^7)(v - \xi^8)(v + \xi^8)(v - \xi^9 )^2(v + \xi^9)^2(v - \xi^{11})(v + \xi^{11})$ \\ 

\smallskip $s_{\rp{2}{\;}}$ & \smallskip $2v^{-4}(v - \xi^4)(v + \xi^4)(v - \xi^6)^2(v + \xi^6)^2(v - \xi^8)(v + \xi^8) $  \\ 

\smallskip $s_{\rp{2}{+}}$ & \smallskip  $-12v^{-4}(v - \xi)(v + \xi^3)^2(v + \xi^5)(v + \xi^7)(v + \xi^9)^2(v - \xi^{11}) $\\ 

\smallskip $s_{\rp{2}{-}}$ & \smallskip $-12v^{-4}(v + \xi)(v - \xi^3)^3(v - \xi^5)(v - \xi^7)(v - \xi^9)^2(v + \xi^{11}) $ \\ 

\smallskip $s_{\rp{3}{+}}$ & \smallskip $v^{-10}(v - \xi^3)^2(v + \xi^3)^2(v - \xi^6)^2(v + \xi^6)^2(v - \xi^9)^2(v + \xi^9)^2 $ \\ 

\smallskip $s_{\rp{3}{-}}$ & \smallskip $v^{-2}(v - \xi^3)^2(v + \xi^3)^2(v - \xi^6)^2(v + \xi^6)^2(v - \xi^9)^2(v + \xi^9)^2 $ \\ 

\smallskip $s_{\rp{4}{\;}}$ & \smallskip $3v^{-4}(v - \xi^3)^2(v + \xi^3)^2(v - \xi^9)^2(v + \xi^9)^2 $ \\ 
\hline 
\end{tabular} %}
\end{center}
\caption{Schur Elements for $\mathcal{H}_{q}(\gtw)$, $q=v^2$, $\xi=e^{2 \pi i /24}$}\label{Schur}
\end{table}

\pagebreak


\begin{thebibliography}{99}
\begin{normalsize}

%%[1] J. van der Geer, J.A.J. Hanraads, R.A. Lupton, The art of writing a scientific article, J. Sci. Commun. 163 (2000) 51Ð59. 

 \bibitem[BC]{BC} Y. Berest, O. Chalykh, \It{Quasi-Invariants of Complex Reflection Groups}, preprint arXiv:0912.4518v1

% [BCP] Bosma,W, Cannon J, and Playoust, C, \It{The Magma Algebra System. I. The user language}, J. Symb. Comput. Vol. 24, No. 3-4, (1997) 235-265. \newline http://magma.maths.usyd.edu.au/magma/
 
 \bibitem[BCP]{Magma}
W. Bosma, J. Cannon, C. Playoust. \textit{The Magma algebra system. I. The user language.} J. Symb. Comput, Vol. 24, No. 3-4 (1997), 235-265. \newline http://magma.maths.usyd.edu.au/magma/
 
 \bibitem[BE]{BE} R. Bezrukavnikov, P. Etingof, \It{Parabolic induction and restriction functors for rational Cherednik algebras}, Sel. Math, Vol. 14, No. 3-4 (2009), 397-425. 
 
 \bibitem[BEG]{BEG} Y. Berest, P. Etingof, V. Ginzburg, \It{Finite dimensional representations of rational Cherednik algebras}, Int. Math. Res. Not, 19 (2003), 1053-1088. 
 
 \bibitem[BMR]{BMR} M. Broue, G. Malle, R. Rouquier, \It{Complex reflection groups, braid groups, Hecke algebras}, J. Reine und Angew. Math, 500 (1998), 127-190. 

\bibitem[C]{C} M. Chlouveraki, \It{On the cyclotomic Hecke algebras of complex reflection groups}, Ph.D. thesis arXiv:0710.0776v1.  

%\bibitem[CH]{CH} T. Chmutova, \textit{Representations of the rational Cherednik algebras of dihedral type}. J. of Algebra, 297 (2006), 542-565.

\bibitem[EM]{EM} P. Etingof, X. Ma, \It{Lecture notes on Cherednik algebras}, arXiv:1001.0432v4. 

\bibitem[ES]{ES} P. Etingof, E. Stoica, \It{Unitary representations of rational Cherednik algebras}, Rep. Theory, 13 (2009), 349-370.  

\bibitem[GG]{GG} I. Gordon and S. Griffeth, \It{Catalan numbers for complex reflection groups}, preprint arXiv:0912.1578v1. 

 \bibitem[GGOR]{GGOR} V. Ginzburg, N. Guay, E. Opdam, and R. Rouquier, \It{On the category $\mo$ for rational Cherednik algebras}, Invent. Math, Vol. 154, No. 3 (2003) 617-651. 

 \bibitem[GP]{GP} M. Geck, and G. Pfeiffer, \It{Characters of Finite Coxeter groups and Iwahori-Hecke algebras}, Oxford University Press, 2000. 
  
\bibitem[M]{M} J. Michel, \It{GAP Manual}, http://www.math.jussieu.fr/~jmichel/ .

\bibitem[R]{R} R. Rouquier, \It{q-Schur algebras and complex reflection groups}, I, Moscow Math. J, Vol. 8, No. 1 (2008), 119-158.

\bibitem[S]{S} Y. Sun, \It{Finite dimensional representations of the rational Cherednik algebra for $G_4$}, J. of Algebra, 323 (2010), 2864-2887.

\bibitem[ST]{ST} G. Shephard and J. Todd, \It{Finite unitary reflection groups}, Canad. J. Math, 6 (1954), 274-304. 

\end{normalsize}
\end{thebibliography}
\end{document}